\documentclass[11pt]{amsart}

 \usepackage[all]{xy}
\usepackage[utf8]{inputenc}
\usepackage[T1]{fontenc}
\usepackage[paper=a4paper, marginpar=2.4cm]{geometry}
\usepackage{amssymb, amsmath}
\usepackage{mathtools}
\usepackage{graphicx}
\usepackage{xcolor}
\usepackage[french, english]{babel}
\usepackage{mathabx}
\usepackage{enumerate, hyperref}

\newcommand{\cc}{\mathbb{C}}

\newcommand{\re}{\mathbb{R}}
\newcommand{\dd}{\mathbb{D}}
\newcommand{\zz}{\mathbb{Z}}
\newcommand{\nn}{\mathbb{N}}
\newcommand{\pp}{\mathbb{P}}
\newcommand{\e}{\varepsilon}

\newcommand{\fr}{\partial}
 \newcommand{\set}[1]{\left\{#1\right\}}
\newcommand{\norm}[1]{\left\Vert#1\right\Vert}
\newcommand{\abs}[1]{\left\vert#1\right\vert}

\newcommand{\pu}{{\mathbb{P}^1}}

\newcommand{\rest}[1]{ \arrowvert_{#1}}
\newcommand{\m}{{\bf M}}

\newcommand{\unsur}[1]{\frac{1}{#1}}

\newcommand{\lrpar}[1]{\left(#1\right)}

\newcommand{\la}{\lambda}

\renewcommand{\propto}{\!\varpropto\!}

\newcommand{\hot}{\mathrm{h.o.t.}}

\newcommand{\inv}{^{-1}}

\DeclareMathOperator{\area}{Area}
\DeclareMathOperator{\supp}{Supp}

\DeclareMathOperator{\Int}{Int}

\DeclareMathOperator{\leb}{Leb}

\DeclareMathOperator{\length}{Length}

\DeclareMathOperator{\diam}{Diam}

\newtheorem{prop}{Proposition} [section]
\newtheorem{thm}[prop] {Theorem}

\newtheorem{lem}[prop] {Lemma}
\newtheorem{cor}[prop]{Corollary}
\newtheorem{theo}{Theorem}

\newtheorem*{problem}{Problem}

\theoremstyle{remark}

\newtheorem{rmk}[prop]{Remark}

\title[Locally biholomorphic Julia sets]{When do two rational functions have locally biholomorphic Julia sets?}

\author{Romain Dujardin}
\address{Sorbonne Universit\'e, Laboratoire de probabilit\'es, statistique et mod\'elisation,    4 place Jussieu, 75005 Paris, France}
\email{romain.dujardin@sorbonne-universite.fr}
 \author{Charles Favre}
\address{CMLS, Ecole Polytechnique, 91128 Palaiseau Cedex, France}
\email{charles.favre@polytechnique.edu}
 \author{Thomas Gauthier}
\address{Laboratoire de Math\'ematiques d'Orsay, B\^atiment 307, Universit\'e Paris-Saclay, 91405 Orsay Cedex, France}
\email{thomas.gauthier1@universite-paris-saclay.fr}

\begin{document}

\begin{abstract}
In this article we address the following question, whose interest was recently renewed 
by problems arising in arithmetic dynamics:  under which conditions does there exist 
a local biholomorphism between the Julia sets of two given one-dimensional rational maps? 
In particular we find criteria ensuring that such a local isomorphism is induced 
by an algebraic correspondence. 
This extends and unifies 
classical results    due to Baker, Beardon, Eremenko, 
Levin,  Przytycki and others.  The proof involves entire curves and positive currents.
\end{abstract}

\maketitle

\section{Introduction}

The problem of determining when two rational maps have the same Julia set has 
been considered by many authors  
in the holomorphic dynamics literature  \cite{baker-eremenko, beardon, dinh, eremenko, levin, levin-przytycki, SS}, 
in relation with some
 classical functional equations. In certain situations (e.g., if the Julia set is the whole sphere), it is
 preferable to ask when two rational maps have the same measure of maximal entropy. 
 The conclusion is that these rational maps satisfy an algebraic relation whose analysis is 
 quite delicate (this  was  explored  in~\cite{pakovich_commuting, ye}). 
 Such rigidity issues have recently played an important role in 
 arithmetic dynamics (see e.g., \cite{baker-demarco,  favre-gauthier, GNY}, and also Remark~\ref{rmk:dmm} below).
 
 In this article we consider the following problem: 
 is any local biholomorphism preserving the Julia set (or the measure of maximal entropy)
 induced by an algebraic correspondence? 
The case of polynomials of the same degree with connected and locally connected 
Julia sets was recently addressed by Luo~\cite{luo}. Here
we deal with general rational maps, possibly of different degrees,
and  obtain a rather satisfactory answer when the maximal entropy measure is (quasi-)preserved
(Theorem~\ref{thm:globalizable}). We are also able to completely solve the    
 problem for polynomials  
 satisfying some mild expansion properties on their Julia sets 
(Theorem~\ref{thm:polynomial_TCE}; see also Remark~\ref{rmk:luo} for a discussion of the relationship between our results and those of~\cite{luo}).   
Related results were obtained in 
\cite{buff-epstein, inou, lomonaco-mukherjee}, under the stronger assumption 
of the existence of a partial analytic conjugacy. 

To be more specific, for a complex rational map $f$ of degree $d\geq 2$, we
 denote by $J_f\subset \pu$ its Julia set and by
 $\mu_f$ its measure of  maximal 
entropy. It is the unique atomless probability measure such that  $f^\varstar \mu_f = d \mu_f$.
Following the terminology of \cite{favre-gauthier},   we say that $f$ is \emph{integrable} if it is either a Chebychev, monomial, or Lattès map. Likewise, we say that $J_f$ is \emph{smooth} if it is equal to $\pu$, 
a circle, or a segment. 
Any integrable rational map has a smooth Julia set, but  there are  many  more   
  examples (see e.g. \cite{eremenko-vanstrien}).

Given two positive measures $\mu_1$ and $\mu_2$, we write $\mu_1\asymp\mu_2$
if  $c\inv\mu_2\leq \mu_1\leq c\mu_2$ for some positive constant 
$c$. In other words,
$\mu_1$ is absolutely continuous with respect to $\mu_2$ and $c^{-1} \leq \frac{d\mu_1}{d\mu_2} \leq c$.
We   use the more precise notation $\mu_1 \asymp_c \mu_2$ 
if there is a need to specify the constant $c$. 
We also write $\mu_1\propto\mu_2$ if $\mu_1$ is proportional to $\mu_2$, that is, if
 there exists $\alpha>0$ such that $\mu_1 = \alpha \mu_2$.

Here is our first  main result:

 \begin{theo}\label{thm:globalizable}
 Let $f_1$ and $f_2$ be   non-integrable 
  rational maps of degree larger than 1, and $U$ be any open subset of $\pu$
 intersecting $J_{f_1}$.
 Let $\sigma \colon U \to \pu$ be any non-constant 
  holomorphic map satisfying $\sigma^\varstar \mu_{f_2} \asymp \mu_{f_1}$ on $U$. When $J_{f_1}$ is smooth we further require that $\sigma^\varstar \mu_{f_2}  \propto\mu_{f_1}$.
  
 Suppose  in addition that there exists a repelling periodic point $p_1$ for $f_1$ such that $\sigma(p_1)$
 is preperiodic under $f_2$.
   
 Then there exist positive integers $a, b\in \nn^*$ and an irreducible algebraic curve 
 $Z\subset \pu \times\pu$ which is preperiodic under $(f_1^a,f_2^b)$ and contains the graph of $\sigma$.
 
In particular we have that $d_1^a   = d_2^b$, and  $\sigma^\varstar \mu_{f_2} \propto  \mu_{f_1}$ in all cases.
  \end{theo}

 In plain words, the  local measure  class preserving morphism $\sigma$ between $J_{f_1}$ and $J_{f_2}$ is induced by an algebraic correspondence between $f_1$ and $f_2$. It is easy to see that such a result cannot be true in the integrable case (see Remark~\ref{rmk:integrable}).
 
   Preperiodic algebraic curves 
under $(f_1,f_2)$ in $\pu\times \pu$ 
were classified by Pakovich \cite{pakovich_curves}: the upshot is that for any such curve 
there exists rational functions $X_1, X_2, Y_1, Y_2$ such that 
the curve is  a component of 
  $Y_1(x) - Y_2(y) = 0$,  $Y_1\circ X_1  = Y_2\circ X_2$, $X_1\circ Y_1  = f_1^n$ and 
 $X_2\circ Y_2 =f_2^n$. If $f=f_1 = f_2$ is a 
 generic map of   degree $d\geq 3$, this curve must be  a component of 
 $\set{f^k(x)  = f^\ell(y)}$ (see~\cite{pakovich_maximal,ye}). We conclude 
  that in this case   any algebraic correspondence is induced by 
  a branch of $f^{-\ell}\circ f^k$.
  In the polynomial case, more precise results were obtained   in~\cite{medvedev-scanlon}.

The proof of Theorem~\ref{thm:globalizable}
is given in Section~\ref{sec:proof_globalizable}. A natural strategy to establish such a 
 result is to use the expansion induced 
by the repelling point $p_1$ to  propagate the local morphism $\sigma$ to the whole of $\pu$. 
The difficulty is that the  limiting objects will be transcendental and highly multivalued.  
Our approach is geometric and based on entire curves and positive   currents: 
notable ingredients include the Ahlfors Five Islands Theorem (in its strong, quantitative form) and 
Siu's decomposition 
 theorem for positive closed currents. 
 
 Another  key tool  in the proof is  a deep rigidity theorem due to Levin~\cite{levin}, which asserts
that, for a given rational map $f$,  
 there are in a sense 
 only finitely many nontrivial local symmetries of  $J_f$ defined in
 a fixed open set $U$ intersecting $J_f$  
 (see \S\ref{subs:levin} for details). A posteriori, Theorem~\ref{thm:globalizable} may be viewed 
 as a refinement of Levin's theorem: these symmetries are local branches of some
 global algebraic symmetry of $f$. 
 
Note that the notion of local symmetry in Levin's theorem is less restrictive than ours: for a non-smooth Julia set it is just a local holomorphic map such that 
$\sigma\inv(J_f) \cap U=J_f\cap U$. In view of this,  it is natural to expect that, for 
non-smooth Julia sets, the assumption that $\sigma$ preserves the measure class of the maximal entropy measure is superfluous in Theorem~\ref{thm:globalizable} (see   Remark~\ref{rmk:TCE} for a related discussion).
More ambitiously, we may ask the following:
\begin{problem}\label{pbm:unfeasible}
Let  $f_1$ and $f_2$ be non-integrable  rational maps, and
  $\sigma$ be a   
 holomorphic map defined in some open set  $U$ intersecting  $J_{f_1}$, 
 such that the equality  $\sigma\inv(J_{f_2})\cap U = J_{f_1}\cap U $ holds if $J_{f_1}$ is not smooth, 
and $\sigma^\varstar \mu_{f_2} \asymp \mu_{f_1}$  when $J_{f_1}$ is smooth.  

Then is $\sigma$ induced by an algebraic correspondence between $f_1$ and $f_2$?
 \end{problem}

 In Sections~\ref{sec:TCE},~\ref{sec:polynomials}, and~\ref{sec:jordan} we take some further steps towards the resolution of this problem.
  We first prove that   the assumption 
 that $\sigma$ maps a repelling point to a preperiodic point can be dropped 
under suitable expansion properties for $f_1$ and $f_2$.   Indeed, in this case
the geometry of the Julia set  (resp. of $\mu_f$) alone is enough 
to detect preperiodic points   
   (see Section \ref{sec:TCE}, and in particular Corollary~\ref{cor:automatic_S3}, for details). 

For polynomials, by using the fact that the maximal entropy measure 
coincides with the harmonic measure of the Julia set (viewed from infinity), we are able to show 
that, under reasonable assumptions, the    measure class preservation 
 $\sigma^\varstar \mu_{f_2}\asymp \mu_{f_1}$ is automatically implied 
by the geometric condition $\sigma\inv(J_{f_2})\cap U = J_{f_1}\cap U$. 
Some  non-trivial potential-theoretic arguments
are developed in Section~\ref{sec:polynomials} to deal with the delicate
  interplay between the local and global properties of the harmonic measure. 

  In Section~\ref{sec:jordan}, we 
  deal with one specific issue concerning polynomial  Julia sets that are Jordan curves: 
  how to make \emph{locally} the distinction between the inside and the outside of $J_f$? 
As a matter of fact, in  Proposition~\ref{prop:inside_out} we prove that if $f_1$ and $f_2$ are polynomials   whose Julia sets are non-smooth Jordan curves, there does not exist a local biholomorphism mapping
 $J_1$ to $J_2$ and mapping the bounded component of  $\cc\setminus J_1$ to the 
 bounded component of  $\cc\setminus J_2$. 
  This is particularly delicate for quasicircles, which have no preferred side from the conformal point of view
  
Combining this to the  results of Section~\ref{sec:TCE} 
yields the following streamlined version of Theorem~\ref{thm:globalizable}, 
which solves the above problem for  polynomials satisfying the so-called 
Topological Collet-Eckmann (TCE)  condition (see Section \ref{sec:TCE} for this notion; Theorem~\ref{thm:polynomial_TCE} is proven in Section~\ref{sec:polynomials}).
   
  \begin{theo}\label{thm:polynomial_TCE}
  Let  $f_1$ and $f_2$  be  polynomials, such that  
 either  of $f_1$ or $f_2$ is non-integrable and 
  satisfies the Topological Collet-Eckmann property. Let 
  $\sigma$ be a   
 holomorphic map defined in some open set  $U$ intersecting  $J_{f_1}$,  such that the equality 
  $\sigma\inv(J_{f_2})\cap U = J_{f_1}\cap U $ holds.
 Then $\sigma$ is induced by an algebraic correspondence between $f_1$ and $f_2$.  
  \end{theo}

\begin{rmk}\label{rmk:dmm}
Our investigations were partially motivated by a question that arose in the work of the first two authors
on the dynamical Manin-Mumford problem for plane polynomial automorphisms~\cite{dmm_automorphisms}.
The context is as follows:

Suppose $f$ is a complex plane polynomial automorphisms of positive entropy, and let $p$ be any hyperbolic fixed point. 
Suppose moreover that there exists a local biholomorphism $\sigma$ from the stable manifold of $p$ to its unstable
manifold that maps the backward invariant current to the forward invariant current (these objects are higher dimensional analogs of the maximal entropy measure). 
Showing that $\sigma$ extends as a global algebraic involution of $\cc^2$ would imply~\cite[Conjecture~1]{dmm_automorphisms}; see Remark 4.4 of op.cit. \qed
\end{rmk}

 \subsection*{Acknowledgements} 
 Nessim Sibony sadly passed away while we were preparing this paper. He was a great promoter of the use of positive currents in holomorphic dynamics, and we  dedicate this paper to his memory. We extend our thanks to Alano Ancona, Guy David and 
 Pascale Roesch for useful conversations. 
 The third author is partially supported by the ANR grant Fatou ANR-17-CE40-0002-01.
  
\section{Proof of Theorem \ref{thm:globalizable}}\label{sec:proof_globalizable}

\subsection{Notation and conventions}
Any positive measure $\mu$ on $\pp^1$ can be locally defined by 
$\mu = dd^c u$ for some subharmonic function $u$. We say that $\mu$ has continuous potentials
when $u$ is continuous. 
 Let $f \colon U \to \pp^1$ be any holomorphic map defined on a connected open subset of the Riemann sphere. 
For any positive measure $\mu= dd^c u$ as before, we locally define 
 $f^\varstar \mu$  by   $f^\varstar \mu = dd^c(u \circ f)$. Alternatively, if $\mu$ gives no mass to points, 
 we may set 
\[f^\varstar \mu\rest{A} = \left({f\rest{f(A)}\inv}\right)_\varstar \mu\] on any Borel set $A$ on which $f$ is injective.

When two rational maps $f_1$ and $f_2$ are given, to ease notation we often write $\mu_i$ for $\mu_{f_i}$, and likewise $J_i$  for the Julia sets, etc. 

We use the   notation  $D(z,r)$ for  the open disk of center $z\in \cc$ and radius $r$ in $\cc$,   write $\dd = D(0,1)$ and identify $\fr\dd$ with   $\mathbb{S}^1  = \re/\zz$. 
Likewise, $B(z,r)$ refers to the  spherical disk of center $z$ and radius $r$ in $\pu$.
We often denote by $C$  a ``constant'' which may change from line to line, independently of some asymptotic quantity which should be clear from the context, and write $a\lesssim b$ for $a\leq Cb$ and $a\asymp b$
 for $C\inv b\leq  a\leq  Cb$.

\subsection{Levin's theorem}  \label{subs:levin}
 
By definition the Julia set of a rational map $f$ is said to be smooth if it contains an open set or if 
on some open subset
it coincides with a smooth arc. 
In this case  it was proved by Fatou that $J_f$ is respectively equal to $\pu$, or it is a circle or a 
segment (see \cite[\S 56 p. 250]{fatou}).  
 
 \begin{thm}[{Levin \cite{levin}}] \label{thm:levin}
Let $f$ be a rational map of degree greater than $1$. 
Suppose there exists  a connected open set $U\subset \pu$ intersecting $J_f$ and 
 an infinite family of holomorphic maps $\sigma_n \colon U\to \pu$ such that 
 $\sigma_n^{-1}(J_f)\cap U  = J_f\cap U$. If in addition $J_f$ is smooth we require that 
 $\sigma_n^\varstar \mu_f \propto \mu_f$.  

If the family  $(\sigma_n)$ is normal and all its  limit functions are non-constant, then $f$ is integrable. 
\end{thm}

In the following, we will refer to a local map $\sigma$ satisfying these assumptions simply as a \emph{local symmetry} of $J_f$.  
Let us also pinpoint an intermediate step in the proof of Theorem~\ref{thm:levin} which will also be useful. 

\begin{lem}[{see \cite[Proposition 1]{levin}}] \label{lem:levin_prop1}
Let $p$ be a repelling fixed point of the rational map $f$. 
Assume that  $\sigma \colon U \to \pu$ is a holomorphic map fixing $p$ such that $\sigma'(p)\neq 0$  and 
$\sigma^{-1}(J_f) \cap U = J_f \cap U$; if in addition $J_f$ is smooth we further require that 
  $\sigma^\varstar \mu_f \propto \mu_f$. Then $\sigma$ and $f$ commute. 
  \end{lem}

\subsection{A normal families lemma}
The normality assumption in Levin's theorem will
 be deduced  from a uniform bound on the Radon-Nikodym derivative $\frac{d\sigma_n^\varstar\mu_f}{d\mu_f}$, 
 thanks to the following 
normal families criterion.

\begin{lem}\label{lem:normal}
Let $U\subset \pu$ be a connected open set, $\nu_1$ a non-zero positive measure on $U$ 
and  $\nu_2$ a positive measure on $\pu$ 
with continuous local potentials. 
If $\sigma_n$ is a sequence of  holomorphic mappings $U\to \pu$  
 such that  $(\sigma_n)^\varstar \nu_2 \asymp_c  \nu_1$ for some uniform $c>0$, 
 then $(\sigma_n)$ is a normal family and all its limiting maps are non-constant. 
\end{lem}

 From Lemma~\ref{lem:normal} and Levin's theorem we get:  
 
\begin{cor}\label{cor:levin}
Let $f$ be a non-integrable rational map of degree greater than 1,   and  
$(\sigma_n)$ be a sequence of 
local symmetries of $J_f$ as in Theorem~\ref{thm:levin}. If in addition 
$(\sigma_n)^\varstar \mu \asymp_c \mu$ for some uniform $c>0$, then 
the family  $(\sigma_n)$ is finite. 
\end{cor}
\begin{proof}[Proof of Lemma \ref{lem:normal}] 
We may assume that $U$ is a disk.  
Note that our assumption implies that $\nu_2$ gives no mass to points, so neither does 
$\nu_1$ on $U$.

For the first assertion,   assume by contradiction that $(\sigma_n)$ is not normal in $U$. Then by the Zalcman reparameterization lemma there exists a sequence $(a_n)$ converging to some $a^\varstar\in U$, a sequence 
$r_n \to 0$ and an extraction $n_j$ such that  the sequence or meromorphic functions
$(\zeta\mapsto \sigma_{n_j}(a_{n_j}+r_{n_j}\zeta))$ 
converges uniformly on compact subsets  to a non-constant entire mapping $\sigma_\infty:\cc\to\pu$.
Since $\nu_2$ gives no mass to points there is a regular value of $\sigma_\infty$ in $\supp(\nu_2)$. 
In particular there is a  disk $D'$  on which 
 $\sigma_\infty$ is univalent, together with a smaller disk 
 $D\Subset D'$ such that $\nu_2(\sigma_\infty(D))>0$. 
 Now  $\sigma_{n_j}$ is univalent   on $a_{n_j} + r_{n_j} D$ for large $j$, and  
 $ \sigma_{n_j}(a_{n_j} + r_{n_j} D)$ converges to $\sigma_\infty(D)$. Therefore if $D$ was further chosen so that $\nu_{2}(\fr(\sigma_\infty(D))) = 0$ 
 we infer that   
 \[
\nu_1(\set{a^\varstar})
=
 \lim_n
 \nu_{1} (a_{n_j} + r_{n_j} D)
\ge c \nu_{2}(\sigma_\infty(D))>0\]
 which is the desired   contradiction.
 
For the second assertion we again argue by contradiction and assume that some subsequence $(\sigma_{n_j})$ 
converges to a constant $a$  on $U$. Let $\varphi$ be a non-negative test function in $U$  such that 
 $ \int \varphi \, d\nu_1 >0$. Let $g_1$ be a subharmonic potential for $\nu_1$ in $U$ and 
$g_2$ be a subharmonic potential for $\nu_2$ defined in a neighborhood of $a$. By assumption, we have
$dd^c(g_2\circ \sigma_n) \ge c \; dd^cg_1$.
Substracting a constant  we may assume that $g_2(a) = 0$. Then we have 
\[
0< c\, \int_U \varphi \, dd^c g_1 = \int_U \varphi\,  dd^c(g_2\circ \sigma_n) = \int_U (g_2\circ \sigma_n)\,  dd^c\varphi
\]
and 
\[
\abs{\int_U (g_2\circ \sigma_n) \, dd^c\varphi}\leq \norm{\varphi}_{C^2(U)} \norm{g_2\circ  \sigma_n}_{L^\infty (\supp (\varphi))}
\]
which tends to zero since $\sigma_n(\supp (\varphi))$ converges to $\set{a}$ and $g_2$ is continuous. This contradiction 
finishes the proof. 
\end{proof}

\begin{rmk}
The continuity of the potential of $g_2$ is essential in the second part of the proof, in particular 
assuming  that $\nu_2$ gives no mass to points is not enough
 to conclude. Indeed the measure $\nu = dd^c (-\log |\log |z||)$  
gives no mass to points and  satisfies $(\sigma_n)^\varstar \nu = \nu$ for $\sigma_n(z) = z^n$, while $z^n$ converges uniformly to 0 in a neighborhood of the origin. 
\end{rmk}

\subsection{Algebraization}

The core of the proof of Theorem~\ref{thm:globalizable}
 is the following  algebraization result, which will later be applied to (generalized) Poincaré-Koenigs linearization mappings.  
At this stage we do not claim   any invariance for  the implied algebraic curve.

\begin{prop}\label{prop:algebraization}
Let $f_1$ and $f_2$ be two  non-integrable  rational maps on $\pu$. 
Assume that $\psi_1$ and $\psi_2$ are 
entire maps $\cc\to \pu$ such that $(\psi_2)^\varstar\mu_{f_2} \asymp  (\psi_1)^\varstar\mu_{f_1}$; if either $J_{f_1}$ or $J_{f_2}$ is smooth then we further require that  $(\psi_2)^\varstar\mu_{f_2} \propto (\psi_1)^\varstar\mu_{f_1}$. 
Define the entire map $\Psi :\cc\to \pu\times \pu$ by $\Psi = (\psi_1, \psi_2)$. Then $\overline {\Psi(\cc)}$  
is an irreducible algebraic curve which is neither a vertical nor a horizontal line.
\end{prop}

Notice that under our assumptions, $J_{f_1}$ is smooth if and only if $J_{f_2}$ is also smooth.

\begin{proof} 
As a preliminary step, let us observe  that
 if $\psi_1$ and $\psi_2$ are rational, then $\Psi (\cc)$ is an algebraic curve 
 since the transcendence degree of $\cc(T)$ over $\cc$ is $1$. Another argument goes by using 
 Remmert's Proper Mapping Theorem and the GAGA principle. 
 So without loss of generality 
 we may assume that $\Psi$
 is transcendental. 

 \noindent{\bf Step 1:}  construction of inverse branches and geometry of the Ahlfors currents.

 Let $\omega_{\pu}$ be a Fubini-Study form on $\pu$, normalized by 
 $\int_\pu \omega_\pu =1$  and set $\omega  = \pi_1^\varstar \omega_{\pu}  + 
\pi_2^\varstar \omega_{\pu}$.  
For any $R>0$, set
\[\area(\Psi(D(0, R))) 
 := 
\int_{D(0, R)}\!\! \Psi^\varstar \omega \text{ , and }\length (\fr \Psi(D(0, R))) 
 :=
\int_{\fr \dd}\!\! |\Psi'(R e^{i\theta})|_\omega R d\theta.\]
Since $\Psi$ is transcendental, $\area(\Psi(D(0, R)))\to\infty$ when $R\to\infty$. 
By the Ahlfors isoperimetric inequality 
(see \cite[\S VI.5]{tsuji} or \cite{brunella})
there exists a sequence $R_j\to \infty$  
such that 
\[\length (\fr \Psi(D(0, R_j))) = \mathrm{o} \big({ \area(\Psi(D(0, R_j)))}\big).\]
Any cluster value 
of the sequence of positive currents
$$T_j:= \unsur{\area(\Psi(D(0, R_j)))} [\Psi(D(0, R_j))]$$ is by definition an Ahlfors current associated to $\Psi$. 
Fix such an Ahlfors current $T$. 
Then $T$ is a positive closed $(1,1)$ current in $\pu\times \pu$ satisfying $\int T\wedge \omega =1$
so there exists $i\in \set{1, 2}$ such that $(\pi_i)_\varstar T\neq 0$, or equivalently $\langle [T], [\pi_i^\varstar \omega_{\pu}]\rangle\neq 0$, 
where $[\cdot]$ denotes the  class in $H^{1,1}(\pu\times \pu)$ and $\langle \cdot, \cdot\rangle$ is the intersection pairing. Without loss of generality we may assume $i=1$.

We now apply Ahlfors' theory of covering surfaces, in the spirit of \cite[\S 7]{bls}, with an additional twist 
inspired from \cite{lamin}. 
Fix any integer $q\geq 5$ and consider $q$ disks
     $D_i$ with disjoint closures,  intersecting $J_1$. 
For every $1\leq i\leq q$, let $ N_j(D_i)$ be the number of 
univalent inverse branches (``good islands'') of $\psi_1$ over   $D_i$ contained in 
$D(0, R_j)$. We label the corresponding  components of $\psi_1\inv(D_i)$ 
as $(\Omega_{i, n})_{n\geq 1}$ in such a way that 
for $1\leq n\leq N_j(D_i)$, $\Omega_{i, n}\subset D(0, R_j)$. Note that at most one of the $\Omega_{i, n}$ contains the origin so we may assume that $0\notin \Omega_{i, n}$. 
Then by Ahlfors' theorem, 
$$\sum_{i=1}^q N_j(D_i)\geq (q-4) { \area_{\pu}(\psi_1(D(0, R_j)))} - h \,
\length_{\pu}(\psi_1(D(0, R_j))),$$ where the area and length are computed with respect to $\omega_{\pp^1}$,
and $h$ is a geometric constant depending only on the disks $D_i$
 (see \cite[Theorem~VI.4]{tsuji}). 
 
Since
 \[
0< \int (\pi_1)_\varstar T \wedge \omega_{\pp^1}
 =
 \lim_{j\to\infty} \frac{\area(\psi_1(D(0, R_j)))}{\area(\Psi(D(0, R_j)))}
 \]
  there exists a constant $C_1$ such that 
 for every $j$, 
 \[\area(\Psi(D(0, R_j)))\leq C_1 \area_{\pu}(\psi_1(D(0, R_j)));\] 
 in particular $\psi_1$ is transcendental. 
 The number of good islands  contained in 
$D(0, R_j)$ whose volume (relative to $\omega$)  is at least $1/2$ is bounded from above
 by $2\area(\Psi(D(0, R_j))$, which is itself bounded by  $2C_1 \area_{\pu}(\psi_1(D(0, R_j)))$.  
Let us discard these components and denote by $ N_j'(D_i)$ the number of remaining ones.  Since these 
components have volume bounded by $1/2$, by Bishop's compactness theorem (see, e.g., \cite[Lemma 3.5]{lamin})
 they form a normal family.  
If $q$ was  chosen so that  $q>4+2C_1$ we infer that 
$$\sum_{i=1}^q N_j'(D_i)\geq (q-4-2C_1) { \area_{\pu}(\psi_1(D(0, R_j)))} - h \,
\length_{\pu}(\psi_1(D(0, R_j))).$$
 
 Extract a further subsequence of   $(R_j)$ (still denoted by $(R_j)$) so that 
 for a fixed $i\in \set{1, \ldots ,q}$
\begin{equation}\label{eq:ahlfors}
N_j'(D_i) \geq \frac{q-4-2C_1}{q}  { \area_{\pu}(\psi_1(D(0, R_j)))} - \frac{h}{q} 
\length_{\pu}(\psi_1(D(0, R_j))),
\end{equation}
and put $D = D_i$ and $\Omega_{n} = \Omega_{i,n}$. 
Let 
$$S_j = \unsur{\area(\Psi(D(0, R_j)))} \sum_{n=0}^{N_j'(D)} \left[ \Psi(\Omega_n) \right] $$
which is a sum of integration currents of graphs over $D$. By   
\eqref{eq:ahlfors}, 
we have that  $S_j\leq T_j$ and we may estimate the mass $\m(S_j) := \int S_j \wedge \omega$ as follows: $$ \m(S_j)  \geq    \int S_j \wedge \pi_1^\varstar \omega  
\geq  \frac{q-4-2C_1}{q} \cdot   \frac{\area_{\pu}(\psi_1(D(0, R_j)))}{ \area(\Psi(D(0, R_j)))}  - 
  \frac{h}{q} \cdot
\frac{\length_{\pu}(\psi_1(D(0, R_j)))}{{ \area(\Psi(D(0, R_j)))} }$$
hence
 $\liminf_j \m(S_j) \geq \frac{q-4-2C_1}{q} \m ((\pi_1)_\varstar T)>0$
and  any cluster value $S$ of the sequence $(S_j)$ satisfies 
$0<S\leq T$.

\medskip 

 \noindent{\bf Step 2:}  using the local symmetries to conclude.  
 
 To simplify notation, write $\mu_1= \mu_{f_1}$ and $\mu_2= \mu_{f_2}$.
For every $n$, define  
 $$\psi_{1, n}\inv := \lrpar{\psi_1\rest{\Omega_n}}\inv\colon D \overset{\sim}{\longrightarrow} \Omega_n$$ and 
 let 
 $\sigma_n := \psi_2\circ  \psi_{1, n}\inv$. 
 By construction $\sigma_n$ is defined in  $D $ with values in  $\pu$.
Writing
   $\psi_2^\varstar \mu_{2}  = h \psi_1^\varstar \mu_{1}$, with $c\inv\leq h\leq c$  for some $c>0$, 
 we infer that  
\begin{align} \label{eq:sigma_n_mu2}
 \sigma_n^\varstar \mu_2&= 
 (\psi_{1,n})_\varstar (\psi_2^\varstar \mu_{2}) 
  =  (\psi_{1,n})_\varstar (h \psi_1^\varstar\mu_1)  \\
  &\notag= \lrpar{h\circ \psi_{1, n}\inv  }   
  (\psi_{1,n})_\varstar  \psi_1^\varstar \mu_1
 =  \lrpar{h\circ \psi_{1, n}\inv   }
 \mu_1
  \asymp_c  \mu_1,
  \end{align}
so by  Lemma \ref{lem:normal}, $(\sigma_n)$ is a normal family
  and its limiting maps are non-constant.
   
   \begin{rmk}\label{rmk:normal_components}
Note that  the  normality of the family $(\sigma_n)$ was already obtained in  Step 1, so
 the  quasi-preservation of the measure is only used 
  to guarantee  that its normal  limit  are non-constant. 
 \end{rmk}
   
Now observe that the maps $\sigma_n$ give rise to local symmetries of $J_{2}$: indeed we can 
pick a subdisk $D'$ intersecting $J_{1}$ on which $\sigma_1$ is univalent, and define a sequence of local symmetries of 
$J_{2}$ by putting $\tau_n = \sigma_n\circ \lrpar{\sigma_1\rest{D'}}\inv$. 
These are holomorphic map from 
$\sigma_1(D') $ to $\pu$  satisfying the relation 
$\tau_n^\varstar\mu_{2}  \asymp_{c^2} \mu_{2}$, where $c$ is as in~\eqref{eq:sigma_n_mu2}.
If in addition $J_{f_1}$ and $J_{f_2}$ are smooth, arguing as in~\eqref{eq:sigma_n_mu2} we  further 
deduce   that $\tau_n^\varstar\mu_{2}  \propto \mu_{2}$.
Thus it follows from   Corollary~\ref{cor:levin}  that the family $(\tau_n)$ is   finite, hence so does the family $(\sigma_n)$. From this we infer that the graphs 
$\Psi(\Omega_n)$ are contained in finitely many graphs over $D$, therefore $S_j$ is an 
 integration current over a fixed finite union of graphs $(\Delta_\ell)$ over $D$, independent of $j$, namely 
 $S_j  = \sum_{\ell} s_{\ell, j}[\Delta_\ell]$.
Extracting a converging subsequence, we get a current 
$S = \sum   s_{\ell }[\Delta_\ell]$ 
supported on the same family of   graphs and from Step 1 we know that $0< S\leq T$. 
    Note that none of these graphs is horizontal because
  $\psi_2^\varstar \mu_{2} \asymp   \psi_1^\varstar \mu_{1}$.
 
With the above notation, fix $\ell$ such that $s_\ell>0$. Then $T\geq s_\ell [\Delta_\ell]$. 
 By Siu's decomposition theorem (see \cite[(2.18)]{MR2978333}) there exists an 
 analytic, hence algebraic, subvariety $\Gamma$ 
of $\pu\times \pu$, extending $\Delta_\ell$, 
 such that $T\geq s_\ell [\Gamma]$.   
 Since $\Delta_\ell\subset \Psi(\cc)$ by construction and 
 $\Delta_\ell\subset \Gamma$, by analytic continuation $\Psi(\cc)$ is contained in $\Gamma$. 
 Therefore we conclude that 
$\overline  {\Psi(\cc)}$ is an   algebraic curve, which is obviously irreducible, and which cannot be 
neither a vertical line because it contains a graph over the first coordinate, nor a horizontal line because 
this graph was shown to be non-horizontal. The proof is complete. 
\end{proof}

\begin{rmk}\label{rmk:transcendental}
Note that if $\psi_1$ is transcendental, then so does $\psi_2$. 
Indeed by Proposition~\ref{prop:algebraization} 
there exists  a polynomial $P\in \cc[x,y]$ such that  $P(\psi_1, \psi_2) =0$. If $\psi_2$ is algebraic, 
it follows that   $\psi_1^{-1}(z_1)$ is finite for any $z_1\in \cc$
which is contradictory.
\end{rmk}

\subsection{Local isomorphisms and Poincaré-Koenigs functions}
Let $f$ be a rational map  of degree $d\geq 2$, and $p$ be a repelling fixed point. 
Denote by 
 $\lambda = f'(p)$ its multiplier.
Then $f $ is linearizable in a neighborhood of $p$, consequently there exists a unique holomorphic map
$\psi_{(f,p)} \colon \cc\to \pu$  such that $\psi_{(f,p)}(0)=p$, $\psi_{(f,p)}'(0)=1$ and for every $\zeta\in \cc$, 
\begin{equation}\label{eq:poincare}
f  \circ \psi_{(f,p)}(\zeta)  = \psi_{(f,p)}(\la \zeta).
\end{equation} 
This map is called the Poincaré-Koenigs linearizing map of $f$.

\begin{lem}\label{lem:generalized_poincare}
Let $\chi \colon (\cc,0) \to (\pu,p)$ be a germ of  non-constant holomorphic map 
satisfying the functional equation 
 \begin{equation}\label{eq:poincare_generalized}
f \circ \chi(\zeta)  = \chi(\kappa \zeta) \text{ for all }\zeta\in \cc \text{ and some } \kappa \in \cc. 
\end{equation} 
If we have $\chi(\zeta) = \beta \zeta^l + O(\zeta^{l+1})$ with $l\ge1$ and $\beta\neq0$, then 
$\kappa^\ell = \lambda$ and $\chi(\zeta) = \psi_{(f,p)}( \beta\zeta^\ell)$.
\end{lem}

Any function $\chi$ satisfying~\eqref{eq:poincare_generalized} will be referred to as a \emph{generalized Poincaré-Koenigs map}. 

\begin{proof}
The expansion of $f \circ \chi(\zeta)$ at the origin together with~\eqref{eq:poincare_generalized} force
 $\lambda =\kappa^l$. 
Locally at $0$, $\psi_{(f,p)}$ is invertible, so that we may consider the holomorphic germ
$\widetilde \chi  := \psi_{(f,p)}\inv\circ\chi$.
Observe that 
\[
\widetilde \chi (\kappa \zeta)
=
\psi_{(f,p)}^{-1} \circ f  \circ \chi  (\zeta) 
=
\lambda \widetilde \chi (\zeta)~.
\]
Expanding $\widetilde \chi$ in power series at $0$ yields
$\widetilde \chi(\zeta)  =  \beta \zeta^\ell$. The proof is complete.
\end{proof}

Recall that for two rational maps $f_i$,  $i=1, 2$ we write $\mu_i = \mu_{f_i}$, $J_i = J_{f_i}$, etc.

 \begin{prop}\label{prop:local_isomorphism_periodic}
 Let $f_1$ and $f_2$ be two non-integrable rational maps of respective degrees
  $d_1, d_2\ge 2$, and   $U$ be any connected open set intersecting $J_{1}$. 
 Suppose $\sigma \colon U \to \pu$ is a non-constant holomorphic map sending a 
 repelling fixed point $p_1$ for $f_1$
to a fixed point $p_2$ for $f_2$. Let $\lambda_1$ and $\lambda_2$ be the respective multipliers of $p_1$ and $p_2$, 
 and set $\ell = \deg_{p_1}(\sigma) \ge 1$. 

Suppose that:
\begin{enumerate}
 \item either $J_{2}$ is not smooth; 
 \item or $J_{2}$ is   smooth and $\sigma^\varstar(\mu_{2}) \propto \mu_{1}$. 
\end{enumerate}
Then the point $p_2$ is repelling, and there exist $a,b\in\nn^*$ such that 
 $\lambda_1^{a \ell} =\lambda_2^{b}$ and   $f_2^b \circ \sigma = \sigma \circ f_1^a$.
Moreover, for $\chi_1 = \psi_{(f_1, p_1)}$, 
  $\chi_2 := \sigma \circ \chi_{1} $ extends to a 
  generalized Poincaré-Koenigs map for $f_2$
 satisfying $\chi_1\inv(J_{1}) = \chi_2\inv(J_{2})$.

 If in Case (1) we further assume: 
 \begin{enumerate}
 \item[(1')] $J_{2}$ is not smooth  and  $\sigma^\varstar \mu_{2} \asymp \mu_{1}$
 \end{enumerate} 
 then we have the identities  $d_1^{ a} = d_2^{b}$
and 
 $\chi_1^\varstar  \mu_1 \asymp \chi_2^\varstar \mu_2$ (resp. $\chi_1^\varstar  \mu_1 
 \propto \chi_2^\varstar \mu_2$ in Case (2)). 
 \end{prop}
 
 \begin{rmk}
 If $p_1$ and $p_2$ are periodic of respective periods $m_1$ and $m_2$, applying this result to 
 $f^{m_1}$ and $f^{m_2}$ we get a similar conclusion, where the relations become 
 $\lambda_1^{a \ell} =\lambda_2^{b}$,   $f_2^{m_2b} \circ \sigma = \sigma \circ f_1^{m_1a}$, and
 $d_1^{m_1 a} = d_2^{m_2 b}$. 
 \end{rmk}

\begin{proof}
Choose local coordinates such that 
  $p_1 = p_2 = 0$ and $\sigma(z) = z^\ell$ for some $\ell \in \nn^*$. Fix any $\ell$-th root $\kappa_2$ of $\lambda_2$.
Then  we can write
 \[f_2\circ \sigma(z)   = \lambda_2 z^\ell + \hot =  \left(\kappa_2 z + \sum_{j\ge 2} a_j z^j\right)^\ell
 \]
and we set $g_1(z):=  \kappa_2 z + \sum_{j\ge 2} a_j z^j$ so that $f_2 \circ \sigma = \sigma \circ g_1$.
Note that $g_1$ is a local isomorphism at $p_1$ 
which locally satisfies  $g_1^{-1}(J_{1}) = J_{1}$   
  in Case~(1) and 
$g_1^\varstar \mu_{1} \propto \mu_{1}$
in Case~(2).
 
 Lemma \ref{lem:levin_prop1} implies that $f_1$ and $g_1$ commute.
  In the linearizing coordinate of $f_1$, the map $\widetilde g_1$ corresponding 
 to $g_1$ is a local biholomorphism satisfying 
 $\widetilde g_1(\lambda_1 \zeta) = \lambda_1 \widetilde g_1(\zeta)$. 
 Expanding $g_1$ in power series, and since $\lambda_1$ is not a root of unity,
 we obtain that $\widetilde g_1$ is linear: $\widetilde g_1(\zeta) = \kappa_2\zeta$. The subgroup generated by $
 \lambda_1$ and $\kappa_2$ in $\cc^\varstar$ must be   
  discrete otherwise by taking sequences $(k_j)$ and $(\ell_j)$ such that 
 $\lambda_1^{k_j}\kappa_2^{\ell_j} \to 1$ we would create an infinite
  normal  family  of local  symmetries of $J_1$ contradicting the fact that $f_1$ is not integrable. 
  It follows that  there is a relation of the form $\lambda_1^a=\kappa_2^{b}$ for some $a\in \nn$ and $b\in \zz\setminus\{0\}$. 
 Since $p_2\in J_2$, $\abs{\lambda_2}\geq 1$ so $\abs{\kappa_2}\geq 1$. 
 Since  $\widetilde g_1$  has infinite order, $\kappa_2$ is not a root of unity so   $a$ is positive.  
This implies that  $b$ is positive as well, hence   $\abs{\lambda_2}> 1$, i.e. $p_2$ is repelling. 
Thus we have shown that    there is a relation  of the form   $\lambda_1^{\ell a} = \lambda_2^{b}$,
 with $a,b>0$, as asserted.   
 Back to the initial coordinates, this means that $f_1^a = g_1^b$ so that
 $\sigma \circ f_1^a= \sigma \circ g_1^b = f_2^b \circ \sigma$. 
 
Now observe    that with $\chi_1 = \psi_{(f_1, p_1)}$ we have 
\[\sigma\circ\chi_{ 1}(\lambda_1^a \zeta) = \sigma \circ f_1^a (\chi_{ 1}(\zeta))   = \sigma \circ g_1^b (\chi_{ 1}(\zeta)) 
  = f_2^b\circ \sigma \circ\chi_{1}(\zeta), \] hence by~Lemma \ref{lem:generalized_poincare}
  locally we have 
  $\sigma\circ\chi_{ 1} (\zeta)=\psi_{(f_2, p_2)}(\zeta^\ell)$.  
  Set  $\chi_2 (\zeta):=\psi_{(f_2,p_2)}(\zeta^\ell)$, which by definition 
  is a generalized Poincaré-Koenigs map.
Locally near the origin we have  $\sigma\circ\chi_{ 1} = \chi_2$, hence
\begin{equation}\label{eq:egaliteJ}
\chi_2\inv(J_2) = (\sigma\circ\chi_{1})\inv(J_2) = \chi_{1}\inv(\sigma\inv(J_2)) =  \chi_{1}\inv(J_1). 
\end{equation}
Since $\chi_2(\lambda_1^a \zeta)   = f_2^b\circ \chi_2(\zeta)$, 
   $\chi_2\inv(J_2)$ is invariant under multiplication by $\lambda_1^a$. The same holds evidently 
   for $\chi_{1}\inv(J_1)$, so~\eqref{eq:egaliteJ} propagates from a neighborhood of $0$ to the whole complex plane.

Now assume that  we are in Case~(1') so that 
$\sigma^\varstar \mu_2 \asymp_c \mu_1$. 
From the relation $f^n_2\circ \sigma = \sigma \circ g^n_1$, for all $n\in\nn$ we obtain
\[
d_2^{bn} \mu_1 \asymp_{c^2} (g^{bn}_1)^\varstar \mu_1,
\]
 so that
  \[
   d_1^{an} \mu_1
   =
   (f_1^{an})^\varstar \mu_1 
   =
  (g_1^{bn})^\varstar \mu_1
  \asymp_{c^2}
  d_2^{bn} \mu_1 \] 
  which implies that 
  $d_1^a= d_2^b$.
  
Locally near the origin we have that 
 \begin{equation}\label{eq:egalitemu}
 \chi_2^\varstar\mu_2  = \chi_{1}^\varstar (\sigma^\varstar\mu_2) \asymp_c \chi_{1}^\varstar \mu_1. 
 \end{equation}
It remains to explain why the  relation~\eqref{eq:egalitemu}  propagates to $\cc$. 
 Write $M_\kappa(\zeta) := \kappa \zeta$, and define positive measures on $\cc$ by 
   $\widetilde \mu_1:=\chi_{1}^\varstar \mu_1$ and  $\widetilde \mu_2:=\chi_2^\varstar\mu_2$.
   We have
\[
M_{\lambda_1^a }^\varstar \widetilde \mu_1
=
(\chi_{1} \circ M_{\lambda_1^a })^\varstar
\mu_1
=
\chi_{1}^\varstar (f_1^\varstar
\mu_1)
=
d_1^a   \widetilde \mu_1, \] and likewise, since $\chi_2\circ M_{\lambda_1^a} = f_2^b\circ \chi_2$ we get 
$$ M_{\lambda_1^a }^\varstar \widetilde \mu_2 
= M_{\lambda_1^a }^\varstar \chi_2^\varstar \mu_2
= \chi_2^\varstar  (f_2^b)^\varstar \mu_2  = d_2^b \widetilde \mu_2 = d_1^a \widetilde \mu_2 .$$
 Therefore $\widetilde \mu_1$ and $\widetilde\mu_2$ are positive measures on $\cc$ satisfying the same 
 relation $M_\Lambda^\varstar \widetilde \mu_i   = D \widetilde \mu_i$ (for $\Lambda = \lambda_1^a$ and $D = d_1^a$) 
 and such that 
 $\widetilde \mu_1\asymp_c\widetilde\mu_2$ in some small 
 disk $D(0,r)$ . Since $M_\Lambda$ is  invertible on $\cc$, it follows that  
 $\widetilde \mu_1\asymp_c\widetilde\mu_2$ globally. Indeed, let $A$ be any Borel set and let $n$ be 
 so large   that $M_\Lambda^{-n}(A)\subset D(0,r)$. Then 
 $$\widetilde \mu_2(A) = D^n \widetilde \mu_2\lrpar{M_\Lambda^{-n}(A)} \leq c D^n \widetilde \mu_1\lrpar{M_\Lambda^{-n}(A)} = c\widetilde \mu_1(A)$$ and similarly for the reverse inequality, so we are done. 
 
 In Case~(2), repeating the same argument with $\sigma^\varstar \mu_2 \propto \mu_1$  
 we arrive 
 at $\widetilde \mu_1\propto\widetilde\mu_2$, and the 
 proof is complete. 
    \end{proof}

\begin{rmk}
This argument is reminiscent from 
 the work of Ghioca, Nguyen and Ye~\cite{GNY}. 
\end{rmk}

\subsection{Conclusion of the proof of Theorem~\ref{thm:globalizable}}
Recall that $\sigma\colon U\to \pp^1$ is a non-constant holomorphic map
such that $\sigma^\varstar \mu_{f_2} \asymp \mu_{f_1}$ (resp. 
$\sigma^\varstar \mu_{f_2} \propto \mu_{f_1}$ when $J_1$ is smooth), and $\sigma(p_1) = p_2$
where $p_1$ is a repelling periodic point for $f_1$ and $p_2$ a preperiodic point for $f_2$. 
Replacing $f_1$ by a suitable iterate, and $\sigma$ by $f_2^k \circ \sigma$ for a suitable $k$
we may suppose that $f_1(p_1)=p_1$ and $f_2(p_2) =p_2$.
Also we write $\ell = \deg_{p_1}(\sigma)$.

 By Proposition~\ref{prop:local_isomorphism_periodic}, the point $p_2$ is repelling and 
 there exist a generalized Poincaré-Koenigs maps $\chi_1$, $\chi_2$   
 associated to  $(f_1,p_1)$ and  $(f_2,p_2)$ such that 
$\chi_{1}^\varstar \mu_1 \asymp \chi_2^\varstar\mu_2$. 
We also have the relations $d_1^a = d_2^b$  for some $a,b\in \nn^*$, and 
 \begin{equation}\label{eq:functional}
 \chi_{1}(\lambda_1^{a} \zeta) = f_1^{a}(\chi_{1}(\zeta)) \text{ and }\chi_2(\lambda_1^{a} \zeta) = 
 \sigma \circ \chi_1(\lambda_1^{a} \zeta) 
 = \sigma \circ f_1^a \circ \chi_1(\zeta)
= f_2^{b}(\chi_{1}(\zeta)).
\end{equation}
Let $\Psi  =   (\chi_1, \chi_2): \cc\to \pu\times \pu$ and $F = (f_1^a, f_2^b)$. 
By Proposition~\ref{prop:algebraization}, 
$\overline{\Psi(\cc)}$ is an irreducible algebraic curve $Z$ which is neither a horizontal nor a vertical line, 
and from~\eqref{eq:functional} we deduce  that 
\[F (\Psi(\zeta)) = \Psi ( \lambda_1^{a} \zeta)\]
so $Z$ is $F$-invariant.

Now note that the first (resp. second) projection $\pi_1\colon \pu \times \pu \to \pu$ (resp. $\pi_2$)
semi-conjugates $F|_Z$ to $f_1$ (resp. to $f_2$). This implies that the measure of maximal entropy
$\mu|_F$ is equal to 
\[
\mu_{F|_Z}=\frac1{\deg(\pi_1)} \pi_1^\varstar \mu_{f_1}=\frac1{\deg(\pi_2)} \pi_2^\varstar \mu_{f_1}
\]
which implies $\sigma^\varstar \mu_{f_2} \propto \mu_{f_1}$ (see e.g., \cite{GNY} for details).
This concludes the proof. \qed
 
 \begin{rmk}  \label{rmk:integrable}  
  {Theorem~\ref{thm:globalizable} fails in the integrable case. }
 
 Indeed, let $f_1$ and $f_2$ be arbitrary Lattès maps, associated to finite branched covers 
 $\pi_i: \cc/\Lambda_i \to \pu$, for some lattices $\Lambda_i$   $i=1, 2$. Let $p_i$ be any point in $\pu$ located 
 outside the critical value locus of $\pi_i$.   Then there is a local measure preserving isomorphism 
 $\sigma$ mapping $p_1$ to $p_2$. Indeed let  $q_i$ be a lift of $p_i$ in $\cc/\Lambda_i$,   
 $\tilde q_i$ be a lift of $q_i$ in $\cc$, and $\tilde \pi_i\colon (\cc,\tilde q_i) \to  (\pu, p_i)$ the natural germ of biholomorphism. Then 
  $(\tilde \pi_i\inv)_\varstar \mu_i$  is proportional to the   
Haar measure of the torus. Therefore,  putting  $\sigma = \tilde \pi_2 \circ\tau \circ (\tilde \pi_1)\inv$, where $\tau$ is the translation mapping $\tilde q_1$ to $\tilde q_2$, we get $\sigma^\varstar\mu_2\propto\mu_1$. 

On the other hand, there is an algebraic correspondence between $f_1$ and $f_2$ when and only
when 
there is an isogeny between the corresponding elliptic curves $\cc/\Lambda_1$ and $\cc/\Lambda_2$, and
$\la_1^a=\la_2^b$ for some $a,b \in \nn^*$ (where $\la_i$ is the derivative of any lift of $f_i$ to $\cc$).

An analogous discussion
can be made in the monomial case. \qed
  \end{rmk}

\section{Local contractions and     preperiodic points}  \label{sec:TCE}

It is  natural to expect that pre-repelling points in the Julia set are   
 geometrically characterized by the existence of a contracting local symmetry. 
We confirm this intuition when $f$ satisfies suitable expansion properties on its Julia set, namely 
when $f$   satisfies the
topological Collet-Eckmann (TCE)
condition. This condition can be defined in a number of equivalent ways, for instance by the following exponential shrinking property: there exists $\lambda>1$ and $r>0$ such that for every $x\in J_f$ and $n\in \nn$,   every connected component $W$ of $f^{-n}(B(x,r))$ satisfies $\diam(W)\leq \lambda^{-n}$. 
  We refer to~\cite{PRS} for a thorough discussion of this notion.

 \begin{thm}\label{thm:TCE}
 Let $f$ be a  non-integrable rational map satisfying the Topological Collet-Eckmann  condition. 
Let  $\sigma \colon U \to \pu$ be a non-constant 
  holomorphic map  satisfying   $\sigma^{-1}(J_f) \cap U = J_f \cap U$,  
  and  furthermore  $\sigma^\varstar \mu_f \propto \mu_f$  if $J_f$ is smooth. 
  
Suppose that 
there exists $p\in J_f \cap U$ such that  $\sigma (p) = p$ and $\abs{\sigma'(p)}<1$. 
Then $p$ is  preperiodic to a repelling point. 
 \end{thm}

\begin{proof} 
Note  that it is enough to show that $p$ is preperiodic: indeed for a TCE map 
all periodic points on the Julia set are hyperbolic. 

To make the main idea more transparent, 
we sketch a proof under the stronger assumption that  $f$ is hyperbolic. 
Then there exists $r>0$ such that for every $n\geq 0$, there is a univalent inverse branch $f_{-n}$ 
of $f^n$ on $B(f^n(p), 2r)$ such that $f_{-n}(f^n(p)) = p$. Reducing $r$ if necessary we may assume that 
$B(p, r)\Subset U$, where $U$ is the domain of definition of $\sigma$. 
By the Koebe distortion theorem, 
$f_{-n}$ has uniformly bounded distortion on $B(f^n(p),  r)$. Therefore 
$f_{-n}(B(f^n(p),  r))$ is up to uniformly bounded distortion a disk centered at $p$ and of radius 
$r_n$, where $r_n$ decreases exponentially with $n$. Let now $k = k(n)$ be the least integer such that 
$C^2(\sigma'(0))^k r < r_n$, where $C$ bounds the distortion of $f_{-n}$ on $B(f^n(p),  r)$ and 
the distortion of $\sigma^k$ on $B(p, r)$. Then $f^n\circ \sigma^{k(n)}$  is a sequence of univalent   symmetries of 
$J$ defined on $B(p, r)$,   with derivative at $p$ bounded away from 0 and infinity.   Levin's theorem 
entails that this sequence is finite, and we conclude that there exists $n_1<n_2$ such that 
$f^{n_1}(p) = f^{n_2}(p)$, as desired. 

If $f$ only satisfies the TCE property the argument is similar, except that we can only map a small neighborhood 
of $p$ to the large scale with bounded degree and along a subsequence of integers. More specifically, the TCE condition of~\cite[p. 31]{PRS} reads as follows. There exists a radius $r>0$ and an integer $\delta$ such that if $W_n$ denotes the connected component  
of $f^{-n}(B(f^n(p), r))$ containing $p$, then there exists a sequence of integers $(n_j)$ of positive lower density such that 
\[\deg\left(f^{n_j}:W_{n_j}\to B(f^{n_j}(p), r)\right)\leq \delta.\]
We claim that if $r$ is small enough,  $W_n$ is simply connected for all $n$.  
Then by \cite[Lemma 2.1]{przytycki-rohde}, $f^{n_j}\rest{W_{n_j}}$ satisfies some
bounded distortion properties.

To prove our claim we make the
following   observations: first, the local structure of holomorphic maps shows that there exists 
$r_0 = r_0(f)$ such that for  $r\leq r_0$, for every $p$,  every component $W$ of $f\inv(B(p, r))$ is simply 
connected. Then $W$ is biholomorphic to a disk in $\cc$ so by the maximum principle, if $U \subset  B(p, r)$ 
is a simply connected open set, $f\inv (U)\cap W$ is simply connected.  Next, by the TCE property 
 there exists $r_1$ such that for every $r\leq r_1$, every $p\in \pu$ 
  and every $n\geq 0$ and every component $W_n$ of 
 $f^{-n}(B(p, r))$ has diameter smaller than $r_0$ (see the Backward Lyapunov Stability 
 condition in~\cite[\S 5]{PRS}). 
 Then the simple connectivity of $W_n$   easily follows by induction.

Reduce $r$ if necessary so that $\sigma$ is well defined and univalent on $B(p, 2r)$, 
and write $\la := \abs{\sigma'(0)}<1$.
By the Koebe distortion theorem, there exist constants $C_1$ and $C_2$ such that 
\[
B\left(p, C_1 \la^k r\right) 
\subset 
\sigma^k(B(p,r))
\subset 
B\left(p, C_2\la^k r\right) .
\]

For any $0<\tau<1$, denote by $W_n(\tau)$ the connected component  of 
$f^{-n}(B(f^n(p), \tau r))$  containing $p$. To simplify notation we write  $W'_n= W_n(1/2)$.

Pick $\alpha <1$, and let 
$k = k(n_j)$ 
be the least integer such that 
$\la^k r \le \alpha  \diam (W'_{n_j})$. 
Then $\alpha \diam (W'_{n_j}) \le \la^{k-1} r$, and we get
\[
B\left(p, \alpha \la C_1\diam (W'_{n_j}) \right) 
\subset 
\sigma^k(B(p,r))
\subset 
B\left(p,  \alpha C_2 \diam (W'_{n_j}) \right)
\]
Now, by \cite[Lemma 2.1 (2.3)]{przytycki-rohde}, we have
$B\left(p,  \alpha C_2 \diam (W'_{n}) \right) \subset W'_{n}$ when   $\alpha$   small enough, 
independently on $n$. 
Furthermore  by \cite[Lemma 2.1 (2.2)]{przytycki-rohde}, when $\tau$ is small enough, then for every $n$, 
$W'_{n}(\tau)\subset B\left(p,  \alpha\la C_1 \diam (W'_{n}) \right)$.
 
 It follows that 
  $f^{n_j}\circ \sigma^{k(n_j)}$  is a sequence of  symmetries of  $J$ defined on $B(p, r)$ which satisfies: 
  \begin{align*}
  f^{n_j}\circ \sigma^{k(n_j)}(B(p, r))
  \subset f^{n_j} (
   W'_{n}) \subset B(f^n(p), r/2), \text{ and }
    \\
  \diam \left(f^{n_j}\circ \sigma^{k(n_j)}(B(p, r))\right)
  \ge
  \diam \left(f^{n_j}(W_{n_j}(\tau))\right)
  \ge \tau r~.
  \end{align*}
The first estimate implies that $f^{n_j}\circ \sigma^{k(n_j)}$
 forms a normal family on $B(p,r)$ and the second that no cluster value of this sequence 
    is constant. At this stage we conclude as in the hyperbolic case:   the sequence $f^{n_j}\circ \sigma^{k(n_j)}$ must be finite, and 
  we find integers $n_{j_1}$ and $n_{j_2}$ such that $f^{n_{j_1}}(p) = f^{n_{j_2}}(p)$. 
\end{proof}

As a consequence we infer that the assumption that $\sigma$ maps a repelling point to a preperiodic point is 
superfluous in Theorem~\ref{thm:globalizable} when $f_2$ satisfies the TCE property. 

\begin{cor}\label{cor:automatic_S3}
Let $f_1$ and $f_2$ be two rational maps and assume $f_2$ is non-integrable and satisfies the TCE property.
Let $\sigma \colon U \to \pu$ be any non-constant 
  holomorphic map satisfying 
  $\sigma^{-1}(J_{2}) \cap U = J_{1} \cap U$ if $J_{f_1}$ is not smooth, 
  and   $\sigma^\varstar \mu_{2} \propto \mu_{1}$ otherwise.
Then $\sigma$ maps any repelling periodic point of 
  $f_1$  to a  pre-repelling point  of $f_2$.
\end{cor}

\begin{proof}
Note that $J_2$ is smooth if and only if $J_1$ is smooth. 
Fix a repelling periodic $p_1$ of $f_1$ of period $k$ outside the critical set of $\sigma$. 
Then   $\sigma\circ f_1^{-k} \circ \sigma\inv$ defines a local holomorphic 
contraction  of $J_2$ at $\sigma(p_1)$, 
which furthermore preserves $\mu_2$ up to a constant  if $J_2$ is smooth, 
 thus the previous proposition gives the result. 
\end{proof}

\begin{rmk} \label{rmk:local_TCE}
The TCE property is detected by the maximal entropy measure: indeed it is equivalent to the property that the measure of small balls satisfies an estimate of the form $\mu(B(x,r))\gtrsim r^\theta$ for some $\theta>0$ and for every $x\in J$ (see~\cite{rivera-letelier}). It is not difficult  to see
that if such an estimate holds for every $x\in U$, where $U$ is any open set  interesting $J_f$, then it holds everywhere (possibly with a different $\theta$). It follows that
 under the assumptions of Theorem~\ref{thm:globalizable}, \emph{$f_1$ is TCE if and only if $f_2$ is TCE}. 
\end{rmk}

\begin{rmk}\label{rmk:TCE} 
Pick any local 
symmetry $\sigma$ of $J_f$. By precomposing with some inverse branch of $f$, we may always
assume that it satisfies $\sigma(U)\Subset U$, so it has an attracting fixed point. 
The proof of Theorem~\ref{thm:TCE}
then gives  the existence of 
integers $n_1<n_2$ and  $k_1<k_2$ such that $  f^{n_1} \circ \sigma^{k_1} = f^{n_2}\circ \sigma^{k_2}$, 
so $\sigma^{k_2-k_1}  = f^{-n_2}\circ f^{n_1}$, and we infer that 
an iterate of $\sigma$ is the restriction of an algebraic correspondence. 

Unfortunately,  the algebraicity of $\sigma$ itself does not seem to follow from this relation, and the only route we know of to algebraicity 
goes through measure class preservation and Theorem~\ref{thm:globalizable}. Still, this
gives additional credit to the  problem stated in the Introduction.
\end{rmk}

\section{The polynomial case} \label{sec:polynomials}

 For  polynomials the maximal entropy measure is determined by the Julia set: 
indeed it coincides with  the harmonic measure of $K_f$ viewed from infinity. 
Thus  in this case it is natural to expect that the  measure class preservation 
hypothesis in Theorem~\ref{thm:globalizable} should follow from weaker property 
$\sigma\inv(J_{2})\cap U = J_1\cap U$. However, since we are working locally, 
some technicalities arise and we are 
only able to confirm this expectation under a mild  additional assumption.

\begin{thm}\label{thm:harmonic_measure}
Let  $f_1$ and $f_2$ be non-integrable polynomials such that:
\begin{enumerate}
\item either $J_1$ and $J_2$ are disconnected;
\item or $J_1$ and $J_2$ are connected and  locally connected.  
\end{enumerate}
Let $U\subset \cc$ be an open subset intersecting $J_1$ and 
 $\sigma:U\to\cc$ be a non-constant holomorphic map such that 
$\sigma\inv(J_{2})\cap U = J_1\cap U$. 
Then there exists $\Omega\subset U$ intersecting $J_1$ on which $\sigma^\varstar \mu_2\asymp \mu_1$. 
\end{thm}

 \begin{rmk}\label{rmk:local_connectedness}  
 Under the assumptions of Theorem~\ref{thm:harmonic_measure}, \emph{$J_1$ is connected (resp. locally connected) iff 
 $J_2$ is connected (resp. locally connected)}, so we could state the assumption only for one of $J_1$ or $J_2$. 
 
 Indeed suppose $J_1$ is disconnected. Then  point components of $J_1$ accumulate the whole Julia set (see \S\ref{subs:disconnected} below), hence
 $J_2 \cap \sigma(U)$ admits  points components, and $J_2$ is disconnected. 
 
 When $J_1$ is locally connected, then it is clear that $J_2 \cap \sigma(U)$ is also locally connected. Since
 $f_2$ is open and  $f_2^n(\sigma(U))\supset J_2$ for large $n$, we conclude that $J_2$ is also connected.
\qed  \end{rmk}
 
We now have all the necessary ingredients for  Theorem~\ref{thm:polynomial_TCE}.

\begin{proof}[Proof of Theorem~\ref{thm:polynomial_TCE}] 
Restricting $U$ if necessary we may assume that $\sigma$ is a biholomorphism, 
and without loss of generality we may assume 
that $f_2$ is non-integrable and TCE. 

When $J_2$ is disconnected, then $J_1$ is disconnected too by the preceding remark, and 
Theorem~\ref{thm:harmonic_measure} implies
$\sigma^\varstar \mu_2\asymp \mu_1$ on some $\Omega\subset U$. 
By Remark~\ref{rmk:local_TCE}, $f_1$ is TCE, and non-integrable. 
Corollary~\ref{cor:automatic_S3} implies that $\sigma$ maps any repelling periodic point of 
$f_1$ to a preperiodic point of $f_2$. 
Thus the result follows from Theorem \ref{thm:globalizable}. 

When $J_2$ is connected, then it is not locally smooth since
$f_2$ is not integrable, and it is locally connected by the TCE property; see~\cite{mihalache}.
Applying the local biholomorphism $\sigma$, we infer that 
 $J_1$ is not smooth, hence $f_1$ is not integrable, and 
by Remark~\ref{rmk:local_connectedness}, $J_1$ is connected and locally connected.
We conclude as in the previous case.
  \end{proof}

 The remainder of this section will be devoted to the proof of Theorem~\ref{thm:harmonic_measure}. 
  It relies on a localization principle for harmonic measure,  which requires different arguments
   in the disconnected and  connected cases; the latter  is the most delicate, a more precise   outline is given at the beginning of \S\ref{subs:connected}.

Let us first introduce  a few standard potential theoretic tools 
(see  \cite{peres-morters}  for a gentle introduction to the probabilistic viewpoint  
on  potential theory,   \cite{doob} for a systematic account, and  \cite{garnett-marshall} for the planar case).
If $\Omega$ is a   domain on the Riemann sphere with non-polar complement,  for  $z\in \Omega$  the harmonic measure 
 $\omega(z, \cdot, \Omega)$ is the measure on $\fr \Omega$ defined by 
 declaring that
 $\int  \varphi(w) \omega(z , dw, \Omega)$ is the value at $z$ of the solution of 
 the Dirichlet problem with boundary values given by   $\varphi\in \mathcal {C}(\fr\Omega)$. It is also the exit distribution of the Brownian motion issued from $z$, that is, if we denote by $B_z$ the Brownian motion issued from $z$ on $\pu$, and 
 $\tau_z = \inf\set{t>0, B_z(t)\in \fr\Omega}$   the hitting time of the boundary (which is a.s. finite), then 
 for a (say closed) subset $E\subset \fr\Omega$, $\omega(z,E, \Omega) = \pp\lrpar{B_z(\tau_z)\in E} $. 
 If $\Omega$ is simply connected,  
 let $\phi: \dd\to\Omega$ be a uniformizing map such that $\phi(0)   = z$ 
 (which is unique up to pre-composition with a rotation). Then $\phi$   
extends radially outside a set of rays of zero capacity, and still denoting by $\phi$ this extension we have 
 $$\omega(z,E, \Omega)  = \omega(0, \phi\inv(E), \dd) = \unsur{2\pi}\abs{\phi\inv(E)}, $$
 where $\abs{\cdot}$ denotes usual arclength. 
 We will only use this fact when 
 $\fr\Omega$ is locally connected, in which case by the Caratheodory theorem 
 $\phi$ extends to a continuous surjection $  \overline\dd\to \overline \Omega$. 
 
If $f$ is a polynomial, the properties of the Green function $G_f$ imply that 
the maximal entropy measure $\mu_f$  
coincides with the harmonic measure of the basin of infinity: $\mu_f= \omega(\infty, \cdot, K_f^\complement)$ (where the complement here is understood in $\pp^1(\cc)$).
  
  \begin{rmk}\label{rmk:luo}  
Theorem~\ref{thm:polynomial_TCE} has some overlap with \cite[Theorem 1.4]{luo} 
(see also~\cite[Prop. 4.5]{mcmullen_motion}), which relies on 
completely different ideas (even if Levin's finiteness theorem also plays a key role there).
In his paper, Luo 
assumes that $f_1$ and $f_2$ are polynomials of the same degree with connected and locally connected Julia 
sets, and all these assumptions are essential. In this setting, his result is stronger than ours since no 
 no additional hyperbolicity assumption is required to guarantee that a periodic point 
 is mapped to a preperiodic point.   Note that by applying his methods, we can obtain the 
 following generalization of \cite[Theorem~1.1]{luo}: \emph{if $M_d$ denotes the degree $d$ Multibrot set, then for $d\neq d'$, an open subset of $\fr M_d$ cannot be biholomorphic to an open subset of $\fr M_{d'}$.}
\end{rmk}
  
 \subsection{Proof of Theorem~\ref{thm:harmonic_measure} in the disconnected case}\label{subs:disconnected}
 The localization principle that we use in the disconnected case is the following: 
  
  \begin{lem}\label{lem:potential_disconnected}
  Let $f$ be a polynomial of degree $\geq 2$. Let   $\Omega$ be a   connected  and simply  
  connected open set    with smooth boundary     such that $\Omega\cap J_f\neq \emptyset$ and 
  $\fr \Omega\cap K_f = \emptyset$. 
  Then for every $z\in \Omega\setminus K_f$, there exists a constant $c>0$ such that 
  $c\inv \mu_f \leq \omega(z,\cdot  , \Omega\setminus K_f) \leq c\mu_f$ on $J_f\cap \Omega$.
  \end{lem}

 This is most likely well-known, however we give a (probabilistic) proof because we have not been able to locate 
it in the literature.

\begin{proof}
Since $K_f$ is full, $\Omega\setminus K_f$ is connected. 
By the Harnack inequality, if $L\subset \Omega\setminus K_f$ then for any $(z,z')\in L^2$, there exists $c = c(L)$ such that 
\begin{equation*} 
c\inv \omega(z,\cdot  , \Omega\setminus K_f) \leq \omega(z',\cdot  , \Omega\setminus K_f) \leq 
c \, \omega(z,\cdot  , \Omega\setminus K_f). 
\end{equation*}
In particular if
the conclusion of the lemma  is true for some $z\in \Omega\setminus K_f$, then it is true for every $z\in \Omega\setminus K_f$. 
Reduce $\Omega$ a little bit to get  a smoothly bounded $\Omega'\Subset \Omega$ with the same properties as $\Omega$ and 
such that $\Omega\cap K_f = \Omega'\cap K_f$. Pick $z\in\Omega'$. 
Let $\tilde \mu = \omega(z, \cdot, \pu\setminus K_f)$. By the Harnack inequality again, $\tilde \mu$ is equivalent to $\mu_f$ with uniform bounds, so it is enough to compare 
$\tilde \mu$ and $ \omega(z,\cdot  , \Omega\setminus K_f)$. 

Denote by $\nu_z$ the restriction of $\omega(z,\cdot  , \Omega\setminus K_f)$ to $J_f\cap \Omega$.
The difference between 
$\tilde\mu\rest{\Omega\cap J_f}$ and  $\nu_z$ 
is accounted for  by the contributions of Brownian 
paths leaving $\Omega$ before reaching $\Omega\cap J_f$. 
If a Brownian path from $z$ eventually hits $\Omega\cap J_f$ without staying in $\Omega$, then it must cross $\fr\Omega'$. 
Let $B_z$ be the Brownian motion from $z$ in $K_f^\complement$ killed when hitting $J_f$, 
and $\tau$ be the hitting time of $\Omega\cap J_f$. Then $\tau<\infty$ if $B_z(t)$ exits  $K_f^\complement$ in  
$\Omega\cap J_f$, and the distribution of $B_z(\tau)$ conditioned to $\tau<\infty$ is 
the harmonic measure $\tilde \mu\rest{J_f\cap \Omega}$. 
 Introduce the following  sequence of stopping times: $T'_0=0$,  and 
by induction    $T_{i+1} = \inf\set{t>T'_i, \ B_z(t)\in \fr\Omega}$ and 
$T'_{i+1} =  \inf\set{t>T_{i+1}, \ B_z(t)\in \fr\Omega'}$, so that $T_1$ is the first hitting time of  $\fr\Omega$, $T'_1$  the first hitting time of  $\fr\Omega'$ after $T_1$, etc. 
  The event 
${\tau<T_1}$ holds if $B_z$ does not reach  $\fr\Omega$ before hitting $J_f$ so the distribution of 
$B_z(\tau)$ conditioned to $\set{\tau<T_1}$ is $\nu_z$. Now conditioned to the event
that $\set{T'_i< \tau < T_{i+1}}$, 
the distribution of $B_z(T'_i)$
   is a certain probability measure $p_i$  on $\fr\Omega'$. Hence  
by the strong Markov property of Brownian motion 
  \cite[\S 2.2]{peres-morters}
the conditional distribution of  $B_z(\tau)$ is   $\nu^{(i)} := 
    \int  \nu_w\,  dp_i(w)$, which satisfies  $c\inv \nu_z\leq \nu^{(i)} \leq c  \nu_z $ for a constant 
    $c$  depending only on $\fr\Omega'$ and $z$. Finally, decomposing the event $\set{\tau<\infty}$ as a disjoint union 
$\set{\tau<\infty} = \bigcup_{i\geq 0} \set{T'_i< \tau < T_{i+1}}$, we express $\tilde \mu\rest{J_f\cap \Omega}$ (normalized 
by $\pp({\tau<\infty})$) as an infinite 
convex combination of $ \nu_z $ and of the $ \nu^{(i)}$, and the lemma follows. 
\end{proof}

\begin{proof}[Conclusion of the proof of the theorem]
Assume   that the assumptions of Theorem~\ref{thm:harmonic_measure} hold, with  $J_1$   
disconnected.

Since point components are dense in $J_1$ 
(see DeMarco-McMullen~\cite{demarco-mcmullen} or  Emerson~\cite{emerson}), 
$J_1$ admits arbitrary small  relatively compact components in $U$.

Therefore we can fix a smoothly bounded simply connected 
open set $\Omega$ intersecting $J_1$, relatively compact in $U$  and  such that 
 $\fr\Omega\cap J_1   = \emptyset$. By choosing $\Omega$ small enough 
 we can further assume that $\sigma$ is a biholomorphism in a neighborhood of $\overline{\Omega}$. 
  Its image $\sigma(\Omega)$ under $\sigma$ satisfies the same properties relatively to $J_2$ and 
 $\sigma(U)$. By the holomorphic invariance of harmonic measure, for $z\in \Omega \setminus K_1$ we have  that 
 $\sigma_\varstar(\omega(z, \cdot, \Omega))  = \omega(\sigma(z), \cdot, \sigma(\Omega))$, 
 so the  property $\sigma^\varstar\mu_2\asymp \mu_1$   follows   from  
 Lemma~\ref{lem:potential_disconnected} applied to $\Omega\cap J_1$ and $\sigma(\Omega\cap J_2)$. 
 \end{proof}

 \subsection{Proof of Theorem~\ref{thm:harmonic_measure} in the  connected case}  \label{subs:connected}
 To establish Theorem~\ref{thm:harmonic_measure} for connected $J_f$ we face several difficulties. The first 
 one is that we need to take care of possible boundary effects in Lemma~\ref{lem:potential_disconnected}: indeed the argument of Lemma~\ref{lem:potential_disconnected} breaks down since 
 we cannot assume that $\Omega\cap J_f = \emptyset$. For this, we uniformize $K_f^\complement$ and use 
 some facts from Caratheodory theory (see \S\ref{subs:local_global_harmonic_measure} 
 as well as the notion of endpoint in \S\ref{subs:endpoints}). 
 The other new difficulty is that  given a small open set $\Omega$ intersecting $J_f$ and 
 $x\in \Omega\setminus J_f$, we need to detect whether $x$ belongs to $K_f$ or not,     by using  only the 
 data given by   $J_f\cap \Omega$. If $J_f$ is not a Jordan curve this can be done by looking at the 
  local topological properties of $J_f$ (the endpoints are also used  here). In the Jordan curve case we cannot distinguish the inside from the outside of $J_f$   from topology, nor even from complex analysis
  when the Julia set is a quasicircle, so a completely different argument needs to be found, which is postponed to the next section  (see Proposition~\ref{prop:inside_out}). 
Note that this study is necessary because  the harmonic measures viewed from the two 
sides of a non-smooth Jordan curve are typically mutually singular (see~\cite[Theorem~VI.6.3]{garnett-marshall}).

\subsubsection{From local to global harmonic measure}\label{subs:local_global_harmonic_measure}
We denote by $\arg(\cdot)$  the argument function    $\cc\setminus \re_-$, with values in $(-\pi, \pi)$. 

\begin{lem}\label{lem:sector}  
For $0<\theta<\pi$ and $0<\delta<1$, 
let $S_{\theta, \delta}\subset \dd$ be the sector defined by 
\[S_{\theta, \delta}  =\set{z ,  \ \abs{\arg(z)}< \theta, \  1-\delta < \abs{z} < 1}.\] 
There exists a constant $c = c(\theta, \delta )$ such that for $\zeta_0 = 1-\delta/2$, if 
  $E$ is   any measurable subset of $\fr\dd\cap \set{z,\  \abs{\arg(z)}\leq   \theta/2}$, 
\begin{equation}\label{eq:harmonic}
\omega(\zeta_0, E, S_{\theta, \delta}) \leq \omega(\zeta_0, E, \dd)\leq  c \omega(\zeta_0, E, S_{\theta, \delta}).
\end{equation}
\end{lem}

\begin{proof}
The first inequality  in~\eqref{eq:harmonic} follows automatically from the fact that $S_{\theta, \delta}\subset \dd$. To prove the second one,  let $\phi:\dd\to S_{\theta, \delta}$ be the uniformisation such that 
$\phi (0) = z_0$ and $\phi'(0)\in \re_+$, which depends only on $(\theta, \delta)$ and 
 extends as a homeomorphism from $\fr\dd$ to $\partial S_{\theta, \delta}$. Then $\omega(z_0, E, S_{\theta, \delta})  = \unsur{2\pi}\abs{\phi\inv(E)}$. 
On the other hand  $\phi\inv \lrpar{\overline {S_{\theta, \delta}} \cap \fr\dd}$  is a closed  circular  
arc $I_0$, 
 and for every sub-arc $I\Subset I_0$, 
  $\phi$ extends by Schwarz reflection  to a biholomorphism in  a neighborhood of $I$. 
  Fix $I  = \phi\inv\lrpar{\set{z\in\fr\dd, \abs{\arg(z)}\leq \theta/2}}$.  Then 
  $\phi: I\to \phi(I) $ distorts lengths by a uniformly bounded amount, so $\abs{\phi\inv(E)}\asymp \abs{E}$, where the implied constant  depends only on $\theta$ and  $\delta$.  
 By the Harnack inequalities, we get that $\abs{E}=2\pi\omega(0, E, \dd)\asymp\omega(\zeta_0, E, \dd)$, and this concludes the proof.
\end{proof}
 
Let $K$ be a full connected and locally connected compact set of $\cc$, containing at least two points. 
Denote by $\phi_K \colon \dd \to K^\complement$ the uniformisation map fixing $\infty$ and tangent to the identity at $\infty$.
Recall that by the Caratheodory theorem, $\phi_K$ extends continuously to a map $\phi_K \colon \overline{\dd} \to K^\complement\cup \partial K$. 
A \emph{crosscut}  of $K$ is an open Jordan arc $C$ in $\cc \setminus K$ such that $\overline{C} = C \cup \{a,b\}$ with $a,b \in \partial K$. 
Note that we allow $a=b$. It follows that  the open set $K^\complement \setminus C$ admits two connected  components, see~\cite[Proposition~2.12]{pommerenke}.

\begin{lem} \label{lem:zdunik2}
Let $K$ be any connected and locally connected full compact subset of the complex plane. 
Let $C$ be any crosscut of $K$, and denote by $W$
the  bounded connected component of $K^\complement \setminus C$. 
Suppose that $\overline{W} \cap \partial K$ is not reduced to a singleton (that is, it is not reduced to 
$\overline C\cap \fr K$). 

Then for any point $z\in W$, and for any  $x \in \partial K\cap \overline{W}\setminus \overline{C}$, there exist a neighborhood $\Omega$ of $x$ and a constant $c>0$ such that 
\[
c\inv \omega(\infty, \cdot, K^\complement) \rest{\Omega}\leq \omega(z, \cdot, W)\rest{\Omega} \leq c\omega(\infty, \cdot, K^\complement)\rest{\Omega}.
 \]
 \end{lem}

\begin{proof}
Lift $C$ to    $\hat{C}:=\phi_K^{-1}(C)\subset\dd$. This is an arc in $\dd$ whose 
  closure in $\overline{\dd}$ intersects $\fr\dd$  
  in    two points $e^{i\theta_0}$ and $e^{i\theta_1}$.
Let $\hat{W}$ be the connected component of $\hat{C} \setminus \dd$ which is mapped to $W$, 
and denote by $I:=[\theta_0,\theta_1]$ the arc $\mathbb{S}^1 \cap \partial \hat{W}$. 
Observe that $\phi_K(I) = \overline{W} \cap \partial K$.  By assumption, $I$
contains $\phi_K^{-1}(x)$, so $I$ is non trivial  and in particular, 
$\theta_0\neq \theta_1$. 

Rotate the situation so that $1 \in I$ and $\phi_K(1) = x$. Fix a sector 
$S = S_{\theta, \delta}$ such that  $S\subset \hat{W}$, and
$ \phi_K^{-1}(x)\subset (-\theta,\theta)$.
Pick any $\zeta \in S$ and any open arc $A$
satisfying $\phi_K^{-1}(x)\subset A\subset \overline{A} \subset (-\theta,\theta)$. 
Choose $\Omega = D(x,r)$ with $r>0$ small enough so that
the closure of $\phi_K^{-1}(\Omega)$ is included in $A$.

Lemma~\ref{lem:sector} now asserts that  $\omega(\zeta, \cdot, S)\rest{A}\asymp 
\omega(\zeta, \cdot, \dd)\rest{A}$. Pushing forward by $\phi_K$ and using 
  the conformal invariance of harmonic measure (see, e.g.,~\cite[Theorem~7.22]{peres-morters}), we get that 
$\omega(z, \cdot , \phi_K(S))\rest{\Omega}\asymp \omega(z, \cdot , K^\complement)\rest{\Omega}$, where 
$z = \phi_K(\zeta)$. Thus we infer that 
\[
\omega(z, \cdot , W)\rest{\Omega}\leq   \omega(z, \cdot, K^\complement )\rest{\Omega} \leq c 
\omega(z, \cdot , \phi_K(S))\rest{\Omega} \leq c \omega(z, \cdot , W)\rest{\Omega},
\] where the 
first and last inequality follow  directly from the inclusions $W \subset K^\complement$ and 
$\phi_K(S)\subset W$ respectively. Finally by the Harnack inequality, 
   $\omega(z , \cdot, K^\complement )\asymp
\omega(\infty , \cdot, K^\complement )$, and the  proof is complete.
\end{proof}

\subsubsection{Endpoints}\label{subs:endpoints}
Assume that $J_f$ is connected and locally connected, and write $\phi_f= \phi_{K_f}$. 
We say that $x = \phi_f(e^{i\theta})$ is an \emph{endpoint} of $J_f$  if there exists a  sequence 
of  intervals $(\theta_n^1, \theta_n^2)$ in the circle $\mathbb{S}^1$, decreasing to $\set{\theta}$ and 
 such that  for every $n$, $\phi(\theta_n^1)   = \phi(\theta_n^2)$. The following lemma is 
  essentially contained in~\cite{zdunik}. 

 \begin{lem} \label{lem:zdunik3}
Let $f$ is a polynomial with connected and locally connected $J_f$  and such that  
 $J_f$ is neither a Jordan curve nor an  interval. Then 
 endpoints are dense in $J_f$. 
 \end{lem}
 
 \begin{rmk}\label{rmk:crosscut} If $x\in J_f$ is an endpoint, then for any $\delta>0$, 
 for large enough $n$ the  image $C$ under $\phi_f$ of the 
 chord joining $\theta_n$ to $\theta'_n$ 
 is a crosscut of $K_f^\complement$ contained in $D(x,\delta)$. 
 \end{rmk}
  
 \begin{proof} 
Recall that the identification of external angles in  $K_f$ can be encoded in the so-called Thurston lamination $\mathcal{L}_f$, 
which is the  lamination by hyperbolic geodesics in $\fr \dd$  such that a leaf joins $\theta$ and $\theta'$ in 
$\mathbb{S}^1$ whenever  $\phi_f(e^{i\theta}) = \phi_f(e^{i\theta'})$ (see \cite[Appendix]{thurston} for details). When 
$J_f$ is not a Jordan curve, this lamination is non trivial, so by backward invariance 
the set of endpoints of leaves of $\mathcal{L}_f$ is dense in $\mathbb{S}^1$. A \emph{gap} of this lamination 
is the closure of a connected component of $\dd\setminus \supp(\mathcal{L}_f)$. A gap $P$ is the closure of the (hyperbolic)
convex hull of its intersection with  $\fr\dd$, and if we write $\partial P = P \cap \fr\dd$, then $\fr\dd \setminus \partial P$ has at 
 least three connected components. 
Since
$f$ is not integrable, it follows from~\cite[Proposition II.6.1]{thurston} (see also the discussion in~\cite[\S 4]{luo})
that the union of gaps is dense in $\overline{\dd}$.

Pick any gap $P$. Since the leaves of the lamination do not cross, for each connected component $I$ 
of $\fr\dd \setminus \partial P$ we can find a gap $P_{I}$ contained in the convex hull of $I$, and $I\setminus \fr P_I$ admits at least $4$ components, so two them, say $I_1$ and $I_2$ satisfy $\abs{{I_j}}\leq \abs{I}/2$. Proceeding inductively, we construct a sequence of disjoint intervals $I_{\e_1\cdots \e_n}$, $\e_i\in \set{1,2}$ such that $\abs{ I_{\e_1\cdots \e_n}}\leq 2^{-n}$. Note that if 
$\theta$ is the decreasing intersection of such a sequence of intervals, then $\phi_f(e^{i\theta})$ is an endpoint. Furthermore by construction these endpoints are separated by crosscuts so they are disjoint. 

This argument thus produces a Cantor set, hence an uncountable set,  of endpoints. In particular there exists an endpoint $x$ 
which does not lie in the post-critical set of $f$, so every preimage of $x$ is an endpoint, and we conclude that endpoints are dense in $J_f$. 
\end{proof}

\subsubsection{Conclusion of the proof of  Theorem~\ref{thm:harmonic_measure} when $J_f$ is not a Jordan curve} \label{sec:no-jordan}
   Let $\sigma$ be as in the statement of the theorem
and reduce $U$ if necessary so that 
$\sigma$ is a biholomorphism in a  neighborhood of  $\overline U$. 
Since by assumption $f_1$ and $f_2$ are not integrable, their Julia set cannot be an interval.

Suppose $J_1$ is not a Jordan curve. Then by Lemma~\ref{lem:zdunik3} (see also Remark~\ref{rmk:crosscut})
there exists $a\in J_1\cap U$ and  a   crosscut $C$ of $K_1^\complement$ with 
$\overline{C} \cap K_1 = \{a\}$ (so that $C$ is a Jordan curve),  and $C\subset U$. Let $W$ be the bounded component of $K_1^\complement \setminus C$ and 
pick $z\in W$. Then $\sigma(C)$ is a crosscut of $J_2$
with $\sigma(\overline{C}) \cap J_2 = \{\sigma(a)\}$. Now by the maximum principle $\sigma(C)$ must be contained in 
$K_2^\complement$ otherwise its interior would be disjoint from $J_2$, which is contradictory.  
By applying Lemma~\ref{lem:zdunik2} to $z$ in $W$ and $\sigma(z)$ in $\sigma(W)$, and using the conformal invariance of harmonic measure,
we conclude that  $\mu_1\asymp \sigma^\varstar\mu_2$ on some open set $\Omega$ 
intersecting $J_1$, 
as was to be shown. \qed

\subsubsection{The case of Jordan curves}\label{subsub:jordan}
Assume now that $J_1$ is a Jordan curve. 
Then $J_1$ has no endpoints in $U$, so that $J_2$
has no endpoints either in $\sigma(U)$.
It follows from Lemma~\ref{lem:zdunik3} that $J_2$ is a Jordan curve as well.

Take a crosscut $C$ of $J_1$ in $U$ such that $\overline{C} \cap J_1$ consists now of two distinct points, and 
denote by $W$ the bounded connected component of $K_1^\complement \setminus C$. 
When $\sigma(W)\subset K_2^\complement$, then the arguments of \S \ref{sec:no-jordan} applies ad litteram.
Proposition~\ref{prop:inside_out} below proves that the possibility that $\sigma(W)\subset K_2$ does not 
occur.
This finishes the proof. 
\qed

 \section{Preserving the sides of a Jordan curve Julia set} \label{sec:jordan}

\begin{prop}\label{prop:inside_out}
Let $f_1$ and $f_2$ be non-integrable polynomials such that $J_1$ is a Jordan curve. Let $U$ be a connected  
open subset intersecting $J_1$ and $\sigma$ be a biholomorphism defined 
on $U$ such that $\sigma(U\cap J_1) = \sigma(U)\cap J_2$. Then $J_2$ is a Jordan curve and $\sigma$ maps  
$K_1\cap U$  
 to $K_2\cap \sigma(U)$.    
\end{prop}

\begin{proof} We already saw in \S\ref{subsub:jordan}
that under the assumptions of the proposition,  $J_2$ is a Jordan curve, too. 
We argue by contradiction and assume that $\sigma$ flips the interior of $J_1$ to the exterior of $J_2$. 

\noindent{\bf Step 1:} reduction to the hyperbolic case.  

There is no critical point of $f_1$ on $J_1$, otherwise by pulling back a neighborhood of the 
corresponding critical value, $J_1$ would have several branches at the critical point, and would not be a Jordan curve. The same applies to $J_2$. Thus the 
unique bounded Fatou component of $f_1$ (resp. $f_2$) must be the basin of an attracting or parabolic fixed point.

If $f_1$ has an attracting point, then it attracts all critical points, and 
 we conclude that $f_1$ is hyperbolic. 
In particular $J_1$ is a quasi-circle, and so is $J_2$. In particular $J_2$ has no cusps, so $f_2$ has no parabolic points, and we infer that $f_2$ is hyperbolic as well.  

If $f_1$ has a parabolic fixed point, then $J_1$ admits a dense set of cusps, 
hence $J_2$ cannot be a quasi-circle. It follows that $f_2$ is not hyperbolic, hence $J_2$
admits a dense set of cusps as well.  
Since all critical points of $f_1$ belong to $\Int(K_1)$, \cite[Theorem~3]{baker-eremenko} implies that 
all cusps are preimages of the cusp at the parabolic fixed point $p$. 
From this and the fact that all critical points are attracted to $p$, all cusps of $J_1$
point inwards, that is, towards  the interior of $K_1$. The same must be true of $J_2$.   However, 
our assumption on   $\sigma$ 
implies that the cusps of $J_2$   should  point  outwards,
which is a contradiction. Therefore $f_1$ and $f_2$ are hyperbolic.

\medskip 

\noindent{\bf Step 2:} using  the  uniform Levin theorem. 

A straightforward compactness argument yields the following uniformity statement in    
 Levin's Theorem: 
\begin{cor}[{of Theorem~\ref{thm:levin}}] \label{cor:uniform_levin1}
Let $f$ be a   rational map with a non-smooth Julia set. Fix $r>0$ and 
$r_1<r_2$. Then there exists  
 $M = M(f,r,r_1,r_2)$ such that for every $x\in J_f$ there are at most $M$ local symmetries $\sigma$ of 
$J_f$ defined in $B(x,2r)$ and such that 
\begin{equation}\label{eq:uniform_levin1}
r_1\leq \diam(\sigma(B(x,  r)))   \leq r_2.
\end{equation}
\end{cor}

For local biholomorphisms between Julia sets, this yields:
\begin{lem} \label{lem:uniform_levin2}
Let $f_1$ and $f_2$  be    rational maps with   non-smooth Julia sets. Then, given 
  $r>0$ and $0<c_1<c_2<1$, there exists  
 $M' = M'(f_1, f_2 ,r,c_1,c_2)$ such that for every $x\in J_1$ there are at most $M'$ local biholomorphisms 
  $\sigma$  defined in  $B(x,2r)$, mapping $J_1$ to $J_2$, and such that 
\begin{equation}\label{eq:uniform_levin2}
c_1\leq \diam(\sigma(B(x,r)))   \leq c_2.
\end{equation}
\end{lem}

\begin{proof}
By the Koebe Distortion Theorem, there exist  $r_1$ and $r_2$  depending only on $r$, $c_1$ and $c_2$ such that for any $\sigma$ as in the statement of the lemma 
 $$B(\sigma(x), 2r_1)\subset  \sigma(B(x,r)) \subset{B(\sigma(x), r_2)}. $$ Likewise, there exists 
 $r_3 = r_3(r, c_1, c_2)$ such that  $\sigma\inv(B(\sigma(x), r_1)) \supset B(x, r_3)$ and finally there exists 
 $r_4 = r_4 (r, c_1, c_2)$ such that $\sigma(B(x, r_3))\supset B(\sigma(x), r_4)$. Now fix such a local biholomorphism $\sigma_0$ and let $y = \sigma(x)\in J_2$. For any other such $\sigma$, we infer that $\tau = \sigma\circ \sigma_0\inv$ is a symmetry of $J_2$, defined in $B(y, 2r_1)$, and satisfying
 $$B(\tau(y), r_4)\subset \tau(B(y,r_1))\subset B(\tau(y), r_2),$$ and from Corollary~\ref{cor:uniform_levin1} we conclude that there are only $M' = M'(f_2, r_1, r_4, 2r_2)$ such maps  $\tau$, and we are done. 
 \end{proof}

 The next result plays a key role in our argument. 
 \begin{lem}\label{lem:uniform1}
Let $f_1$ and $f_2$ be non-integrable hyperbolic polynomials such that $J_1$ and $J_2$ are Jordan curves. Let $U$ be a connected  
open subset intersecting $J_1$ and $\sigma$ be a biholomorphism defined 
on $U$ such that $\sigma(U\cap J_1) = \sigma(U)\cap J_2$.

There exists a constant $B$ depending only on $f_1$, $f_2$ and $\sigma$ such that the following holds. 
For any periodic point $p_1\in J_1$ of period $k_1$, 
there exists a local biholomorphism $\tilde\sigma$
defined in a neighborhood of $p_1$ such that $\tilde \sigma(p_1) = p_2$ is periodic under $f_2$ and  
$\tilde\sigma \circ f_1^{k_1b}= f_2^{b'} \circ \tilde\sigma$
for some $b\le B$ and some integer $b'$.
 \end{lem}
 
 \begin{proof}
 Since $f_1$ and $f_2$ are hyperbolic, we may suppose that $|f'_1| >1$ on $J_1$, and
$|f'_2| >1$ on $J_2$. We may also find $r = r(f_1)>0$ such that 
for any point $p\in J_1$ and any $q \in f_1^{-n}(p)$ there
exists a univalent branch $f_{1,-n}$ of $f_1^n$ 
defined in $B(p,2r)$ and mapping $p$ to $q$ (we fix $\rho>0$ such that  the analogous property holds for $f_2$). 

For $p_1$ as in the statement, choose any integer $N\in \nn^*$
and a univalent inverse branch $f_{1,-N}$ of $f_1^N$ defined on $B(p_1,2r)$
with values in $U$. For notational ease we denote by $f_1^{-k_1n}$ the branch of 
$(f_1^{k_1n})\inv$ fixing  $p_1$. 

Write $\kappa = |(f^{k_1})'(p_1)|>1$  and for each $n$ consider the map 
$F_n := \sigma \circ f_{1,-N}\circ f_1^{-k_1n}$. This map is defined
on $B(p_1,2r)$, and by the Koebe Distortion Theorem, it satisfies
\[
B(F_n(p_1), c'_1 \kappa^n)
\subset
F_n(B(p_1,r))
\subset
B(F_n(p_1), c'_2 \kappa^n)
\]
for some uniform constants $c'_1 < c'_2$ depending only on $f_1$ and $\sigma$.

For any $n$, let $m= m_n$ be the largest integer such that
the diameter of $f_2^m (F_n(B(p_1,r))$ is bounded by $\rho$.
Note that $n\mapsto m_n$ is non-decreasing. 
 Write $q_m= f_2^m (F_n(p_1))$.
Let $f_{2,-m}$ be the univalent branch defined on $B(q_m, 2\rho)$ 
and mapping $q_m$ to $F_n(p_1)$. 
Set $L=\sup_{J_2} |f_2'|$.  
By Koebe distortion, we get
\[
f_{2,-m} \lrpar{B\lrpar{q_m,\frac{\rho}{L}}} \supset B\left(F_n(p_1), \frac{\rho}{4L|(f_2^m)'(F_n(p_1))|}\right)~.
\]
Now observe that by maximality of $m$,  $f_2^m(F_n(B(p_1,r)))$ is not included in $B(q_m,\rho/L)$
so that
$\frac{\rho}{4L|(f_2^m)'(F_n(p_1))|} \le c'_2 \kappa^n$. Thus by applying Koebe to $f_2^m$ on the disk $B(F_n(p_1), c'_1\kappa^n)$, 
we conclude that $\tau_n := f_2^m \circ F_n$ is a sequence of univalent maps
defined on $B(p_1,2r)$ such that
\[
\frac{c'_1}{16 Lc'_2} \rho
\le
\diam (\tau_n (B(p_1,r)))
\le
\rho,
\]
and $\sigma_n(J_1\cap B(p_1,2r)) \subset J_2$. 

Lemma~\ref{lem:uniform_levin2} yields an integer $B$ depending only of $f_1$, $f_2$ and $\sigma$, and a pair of integers $0< n < n' \le B$ such that  
$\tau_n = \tau_{n'}$. Expanding this equality gives 
$$f_2^{m_n} \circ \sigma \circ f_{1,-N}\circ f_1^{-k_1n} = 
f_2^{m_{n'}} \circ \sigma \circ f_{1,-N}\circ f_1^{-k_1n'},  $$
that is, 
$$\tilde \sigma = f_2^{m_{n'}- m_n} \circ \tilde \sigma \circ f_1^{-k_1(n'-n)} \text{, where } 
\tilde \sigma =f_2^{m_n} \circ \sigma \circ f_{1,-N}\circ f_1^{-k_1n},$$ and the result follows.
 \end{proof}

\begin{rmk}
By pushing the argument further it is possible to prove that $b'\leq k_1 B'$ for some uniform $B'$.  
\end{rmk}

\noindent{\bf Step 3:} multipliers and smooth rigidity for  expanding maps on the circle. 

In this paragraph we  
prove some rigidity results  for Blaschke products based on periodic points multipliers. 
If $p$ is a periodic point of period $n$, its Lyapunov exponent is by definition $\unsur{n}\log\abs{(f^n)'(p)}$. 
 The following lemma is presumably  well-known. We provide a proof for the convenience of the reader.

\begin{lem}\label{lem:rigidity1}
Let $g$ be a uniformly hyperbolic  Blaschke product   of degree $d$  
such that $J_g = \fr\dd$.
Assume that  the Lyapunov exponents of periodic points of $g$ take only one value $\log e$. 
Then $e=d$ and  $g$ is   conjugate to $z\mapsto z^d$ by a Möbius transformation. 
\end{lem}
 
 \begin{proof}
 Since $g$ is uniformly expanding, any ergodic invariant measure   has a positive Lyapunov exponent. 
By approximating it by periodic orbits we infer that 
this  Lyapunov exponent   is equal to $\log e$.   
Applying  the dimension formula (see e.g.~\cite[Theorem~11.4.1]{pu_book}) to the unique measure of maximal entropy $\mu_g$, 
we get that 
\[1\geq \mathrm{HD}(\mu_g) = \frac{h_{\mu_g}(g)}{\chi_{\mu_g}(g)} = \frac{\log d}{\log e}, \text{ hence } \log e\geq \log d.\]
Likewise, applying it to the unique  smooth invariant measure $\nu$, we get 
 $$1= \mathrm{HD}(\nu)  = \frac{h_{\nu}(g)}{\chi_{\nu}(g)} \leq \frac{\log d}{\log e}, \text{ hence } \log e\leq \log d.$$
 From this we conclude that $e=d$ and that $\nu$ is the measure of maximal entropy.
Then, if we let $h$ be the conjugacy between $g$ and $M_d$, $h\circ g = M_d\circ h$,  by uniqueness of the maximal entropy 
 measure, we infer that $h_\varstar \nu = \leb_{\fr \dd}$, so $h$ is smooth.
 Finally, by \cite[Theorem~4]{shub-sullivan_circle},  $g\rest{\fr\dd}$ is conjugate to $z^{d}$ by a Möbius transformation, 
and the proof is complete. 
  \end{proof}

\begin{lem}\label{lem:specification1}
For any uniformly expanding $C^{1}$-map of the circle, the closure of the set of Lyapunov exponents of periodic orbits 
is an interval.
\end{lem}

From the two previous lemmas we immediately get:

\begin{cor}\label{cor:specification_rigidity1}
Suppose $g$ is a uniformly hyperbolic  Blaschke product of degree $d$ whose set of Lyapunov exponents of periodic orbits is discrete. 
Then $g$ is Möbius conjugate to $z\mapsto z^d$. 
\end{cor}

\begin{proof}[Proof of Lemma~\ref{lem:specification1}] 
It is enough to show that if there are two periodic orbits $x_1$ of period $n_1$, and  $x_2$ of period $n_2$, of
respective Lyapunov exponents $\chi_1$ and $\chi_2$, then there is a 
periodic orbit whose Lyapunov exponent is approximately $\unsur{2}(\chi_1+\chi_2)$.  This follows
from the periodic specification property, which holds for any expanding map of the circle.

Choose $\eta >0$, and pick 
$\e>0$ so small     that  
\[\chi_1- \eta \le \frac1{n_1} \log \abs{ (f^{n_1})'(x) } \le \chi_1+ \eta\] (resp.  
 $\chi_2- \eta \le\frac1{n_2} \log \abs{ (f^{n_2})'(x) }\le \chi_2+ \eta$)
for any $x \in \fr\dd$ such that $d(x,x_1) \le \e$ (resp. $d(x,x_2)\le \e$).

By the periodic specification property, there exists an integer $M\ge 1$ such that for any $q, q'\ge 1$ 
there is a periodic orbit $x, f(x), \cdots, f^{qn_1+q'n_2+2M}(x) = x$ such that 
$d(f^k(x),f^k(x_1)) \le \e$ for all $ 0 \le k \le q n_1-1$, and 
$d(f^{k+qn_1+M}(x),f^k(x_2)) \le \e$ for all $ 0 \le k \le q'n_2-1$. 

The Lyapunov exponent of $x$ then satisfies:
\begin{align*}
\chi &\le  \frac1{qn_1+q'n_2+2M} \left(n_1q (\chi_1+\eta)+ n_2q' (\chi_2+\eta)+  (2M) \log (\sup |f'|)\right)
\\
\chi &\ge  \frac1{qn_1+q'n_2+2M} \left(n_1q (\chi_1-\eta)+ n_2q' (\chi_2-\eta) +(2M) \log (\inf |f'|)\right)
\end{align*}
So if we choose $q$ and $q'$ very large compared to $M$ and satisfying $qn_1= q'n_2$,
we conclude that $\chi$ is very close to $\unsur{2}(\chi_1+\chi_2)$, as announced. 
\end{proof}

\medskip

\noindent{\bf Step 4:} uniformization and conclusion.

We return to the original situation. Let $f_1$ and $f_2$ be non-integrable hyperbolic polynomials such that $J_1$ and $J_2$ are Jordan curves. Let $U$ be a connected  
open subset intersecting $J_1$ and $\sigma$ be a biholomorphism defined 
on $U$ such that $\sigma(U\cap J_1) = \sigma(U)\cap J_2$ and
$\sigma(K_1\cap U) \subset K_2^\complement$.

Let $\phi_1\colon\dd\to  \Int(K_1)$ be a uniformization with $\phi_1(0)$ the attracting fixed point of $f_1$, and 
$\phi_2\colon\dd\to \pu\setminus K_2$ be a uniformization with $\phi_2(0) = \infty$. 
Since $J_1$ and $J_2$ are Jordan curves, both $\phi_1$ and $\phi_2$ extend to 
respective homeomorphisms 
$\fr\dd\to J_1$ and $\fr\dd\to J_2$. Actually, since $J_1$ and $J_2$ are quasi-circles, these homeomorphisms are 
bi-Hölder (indeed they are quasi-symmetric, see \cite[Chapter 5]{pommerenke}). Let $g_1 = \phi_1\inv\circ f_1\circ 
\phi_1$ (resp. $g_2 = \phi_2\inv\circ f_2\circ \phi_2$). Up to conjugating by a rotation, $g_2(z)  = M_{d_2}(z) := z^{d_2}$. 
Observe that $g_1(z)$ extends to a Blaschke product of degree $d_1$. Indeed $g_1$ extends by Schwarz 
reflexion to a rational map which satisfies $g_1(\dd) = \dd$ and $g_1(\dd^\complement) \subset \dd^\complement$, 
hence $\dd$ is totally invariant so $g_1$ is a  Blaschke product of the same degree as $f_1$, which is uniformly hyperbolic on $\fr\dd$ because no critical orbit of $g_1$ approaches $\fr\dd$. 

We claim that the set of Lyapunov exponents of periodic orbits of $g_1$ is discrete. Taking this claim for granted,   by Corollary~\ref{cor:specification_rigidity1}, we obtain that $g_1$ is Möbius conjugate
 to $M_{d_1}$
thus $f_1$ admits a totally invariant fixed point in $K_1$, so it is integrable, which is a contradiction.

To justify our claim, we proceed as follows.  Pick any periodic point $q_1 \in \fr\dd$ for $g_1$ of period $k_1$. 
Then $p_1:= \phi_1^{-1}(q_1)$ is $f_1$-periodic of the same period, and Lemma~\ref{lem:uniform1} implies
the existence of a local biholomorphism $\tilde\sigma$ sending $p_1$ to a periodic point $p_2$ for $f_2$ such that 
$\tilde\sigma \circ f_1^{k_1b}= f_2^{b'} \circ \tilde\sigma$ for some $b\le B= B(f_1,f_2,\sigma)$ and some $b'\ge1$.
By the Schwarz reflection principle, the map 
$\tau:=\phi_2\inv \circ \tilde\sigma \circ \phi_1 $ is defined in a neighborhood of $q_1$
and satisfies
$\tau \circ g_1^{k_1b}= g_1^{b'}\circ\tau  = M_{d_2}^{b'} \circ \tau$. Computing derivatives, this
implies that the Lyapunov exponent of $g_1$ at $q_1$
belongs to $\bigcup_{1\le b\le B}\frac{\nn}{b} \log d_2$ which is a discrete set. 

The proof is complete.
\end{proof}


\bibliographystyle{plain}
\bibliography{bib-symmetry}
\end{document}